\documentclass[a4paper]{article}

\usepackage{amsmath, amssymb, amsthm}
\usepackage{lineno,hyperref}
\usepackage{enumerate}

\newtheorem{thm}{Theorem}[section]

\newtheorem{cor}[thm]{Corollary}

\newtheorem{prop}[thm]{Proposition}
\newtheorem{lem}{Lemma}[section]
\newtheorem{rem}{Remark}[section]
\newtheorem{Def}{Definition}[section]

\allowdisplaybreaks

\begin{document}
	
		\title{Globality of the DPW construction for Smyth potentials in the case of \({\rm SU}_{1,1} \)}
		\author{Tadashi Udagawa}
	    \date{}
	    \maketitle
		
		\begin{abstract}
		We construct harmonic maps into \({\rm SU}_{1,1}/{\rm U}_1 \) starting from Smyth potentials \(\xi \), by the DPW method. In this method, harmonic maps are obtained from the Iwasawa factorization of a solution \(L \) of \(L^{-1} dL = \xi \). However, the Iwasawa factorization in the case of a noncompact group is not always global. We show that \(L \) can be expressed in terms of Bessel functions and from the asymptotic expansion of Bessel functions we solve a Riemann-Hilbert problem to give a global Iwasawa factorization. In this way we give a more direct proof of the globality of our solution than in the work of Dorfmeister-Guest-Rossman \cite{DGR2010}, while avoiding the general isomonodromy theory used by Guest-Its-Lin \cite{GIL20151}, \cite{GIL20152}.
		\end{abstract}
		\vspace{10pt}


	
	\section{Introduction}
	The DPW method (the generalized Weierestrass representation) is a way to construct harmonic maps into symmetric spaces, that was developed by J. Dorfmeister, F. Pedit, H. Wu (see \cite{DPW1998}). The Gauss map of a constant mean curvature (CMC) surface in \(\mathbb{R}^3 \) or \(\mathbb{R}^{2,1} \) is a harmonic map, thus the DPW method is useful for studying constant mean curvature (CMC)-surfaces. The Riemann metric of a CMC-surface satisfies the classical Gauss-Codazzi equation and then the DPW method is also useful for studying certain nonlinear differential equations. In the DPW method, a CMC-surface corresponds to a matrix-valued 1-form and we call this 1-form the holomorphic potential. The properties of CMC-surfaces and Riemann metrics such as symmetry can be expressed as conditions on the holomorphic potential. \vspace{3mm}
	
	The studying of CMC-surfaces with symmetry is an interesting topic and in terms of surface theory there are well-known examples such as the Delaunay surface \cite{E1993} and the Smyth surface \cite{S1993}. 
	In terms of the DPW method, there are many articles about the Delaunay surface \cite{DH1998}, \cite{DH2000}, \cite{K2004}, \cite{SKKR2007} and the Smyth surface \cite{BI1995}, \cite{O2017}, \cite{TPF1994}. In this paper, we investigate the Smyth potential
	\begin{equation}
		\xi = \frac{1}{\lambda} \left( \begin{array}{cc}
			0 & z^{k_0 } \\
			z^{k_1 } & 0
		\end{array} \right)dz,\ \ \ k_0, k_1 \in \mathbb{Z}_{\ge -1 }, \nonumber
	\end{equation}
	on \( \mathbb{C}^* \). The Smyth potential satisfies the homogeneity condition \(\xi(ze^{i a},\lambda e^{ib} ) = T^{-1} \xi(z,\lambda) T \), where \(a,b \in \mathbb{R} \) and \(T \in {\rm U}_1 \subset {\rm SU}_2 \). For the case \(k_0 = 0,\ k_1 \in \mathbb{N} \), the Smyth potential gives the \((k_1+2) \)-legged Smyth surface in \(\mathbb{R}^3 \), which has reflective symmetry with respect to \(k_1 + 2 \) geodesic plane and the metric is independent of \(\theta \) in \(z = re^{i\theta } \). This symmetry is obtained from the homogeneity condition. For the case \(k_1 = -1 \), the homogeneity condition can be interpreted from the perspective of quantum cohomology (see \cite{DGR2010} for details). We mainly discuss this case, following Dorfmeister, M. Guest, W. Rossman in \cite{DGR2010}, in order to give a more direct argument without relying on results on Painlev\'e equations. The ``generic case'' where \(k_0,k_1 \in \mathbb{Z}_{\ge 0} \) is easier, and can be treated by the same method. \vspace{3mm}
	
	For \(k_0, k_1 \in \mathbb{Z}_{\ge -1} \), we consider the Smyth potential corresponding to a CMC surface in the Minkowski space \(\mathbb{R}^{2,1 } \), then the Riemann metric satisfies the sinh-Gordon equation
	\begin{equation}
		u_{t \overline{t } } = e^{2u } - e^{-2u },\ \ u = u(|t|) : U \rightarrow \mathbb{R}, \nonumber
	\end{equation}
	where \(t = \frac{2}{k_0 + k_1 + 2 } z^{\frac{k_0 + k_1 + 2 }{2 } } \) and \(U \subset \mathbb{C}^* \) is an open subset (see section 2 of \cite{DGR2010}). \vspace{3mm}
	
	In the DPW method, we need the Iwasawa factorization of a certain matrix-valued function obtained from the holomorphic potential. In the case of \(\mathbb{R}^3 \) the Iwasawa factorization is always globally-defined on the whole domain. However, in the case of \(\mathbb{R}^{2,1} \) the Iwasawa factorization is not always globally-defined.
	Our goal in this paper is to give a suitable dressing transformation after which we have a global Iwasawa factorization. In this paper, we transform the scalar equation which corresponds to the Smyth potential to the Bessel equation. The properties of the Bessel functions are well-known and we can use them to investigate the matrix-valued function which corresponds to the Smyth potential. In particular, we use the asymptotic expansion of the Bessel functions and prove the existence of the global Iwasawa factorization. As a corollary, we also show the existence of the global solution of the sinh-Gordon equation. \vspace{3mm}
	
	The existence of the global Iwasawa factorization was proved in \cite{DGR2010}, \cite{GIL2020}, but the proof was not obtained from the DPW method directly. In \cite{DGR2010}, the authors proved the existence of the Iwasawa factorization near \(z = 0 \) by using the DPW method, then appealed to results of \cite{FIKN2006}, \cite{IN2011}. In \cite{GIL20151}, \cite{GIL20152}, \cite{GIL2020}, the authors gave a more detailed proof (and a more general result) by combining isomonodromy theory and p.d.e. theory. On the other hand, in this paper we use the DPW method and properties of Bessel functions. We do not need results for p.d.e. theory or isomonodromic deformation theory. Thus, our proof is more direct than \cite{DGR2010} and shorter than \cite{GIL20151}, \cite{GIL20152}, \cite{GIL2020}. \vspace{3mm}
	
	In terms of physics, the relation between \(N = 2 \) supersymmetric quantum field theories and the tt*-Toda equation was introduced by S. Cecotti and C. Vafa (see \cite{CV1991}). Since the sinh-Gordon equation is a special case of tt*-Toda equation, Smyth potentials have some interpretation from physics. In particular, for the case \(k_0 \in \mathbb{Z}_{\ge 0 },\ k_1 = -1 \) the CMC-surface in \(\mathbb{R}^{2,1 } \) obtained from the Smyth potential is related to the quantum cohomology of \(\mathbb{C} P^1 \) i.e. the ordinary differential operator \((\lambda q \partial_q )^2 - q \), which gives the quantum differential equation, can be identified with the Maurer-Cartan form of the resulting surface by a suitable gauging (\cite{DGR2010}). For the general case, the ``magical solution'' of tt*-Toda equation predicted by Ceccoti and Vafa is expected to be globally-defined on \(\mathbb{C}^* \). \vspace{3mm}
	
	This paper is organized as follows: In section 2, we define loop groups and review the DPW method. We review the relation between holomorphic potentials and the surface data of the resulting surfaces. In section 3, we consider a canonical solution \(L \) obtained by solving
	\begin{equation}
		L^{-1 } dL = \frac{1}{\lambda} \left(
		\begin{array}{cc}
			0 & z^{k_0} \\
			z^{k_1} & 0 
		\end{array}
		\right)dz\ \ {\rm on}\ \mathbb{C}^*, \nonumber
	\end{equation}
	where \(k_1 = -1\) i.e. the quantum cohomology case. We follow \cite{DGR2010} and choose a suitable dressing action \(\gamma_0 \) such that the Iwasawa factorization of \(\phi = \gamma_0 L \) is defined near \(z = 0 \). In section 4, we investigate the asymptotic behaviour of \(\phi \) at \(z = \infty \) (to be precise, at \(\lambda^{-1 } z^2 = \infty \)) from the asymptotic expansion of the Bessel functions. In section 5, we consider the Birkhoff factorization of
	\begin{equation}
		g(z,\lambda) = \overline{\phi(z,\overline{\lambda }^{-1} ) }^t \left( \begin{array}{cc}
			1 & 0 \\
			0 & -1
		\end{array} \right) \phi(z,\lambda), \nonumber
	\end{equation}
	which is equivalent to the Iwasawa factorization of \(\phi \). The Birkhoff factorization can be formulated as a Riemann-Hilbert problem, thus we can use results for the Riemann-Hilbert problem. We combine these results and the asymptotic expansion that we proved in section 4 and prove the existence of the global Birkhoff factorization. Here we make use of the ``Vanishing Lemma'' for Riemann-Hilbert problems (Corollary 3.2 of \cite{FIKN2006}). As a corollary, we obtain the global solution of the sinh-Gordon equation. In this paper, we give calculations only in the case \(k_0 = 3, k_1 = -1 \) for the reason that the case \(k_0 \in \mathbb{Z}_{\ge 0}, k_0 = -1 \) can be transformed to the case where \(k_0 \) is fixed, and in the case \(k_0 = 3 \) our calculations become easier. In \cite{DGR2010}, the case \(k_0=0 \) was chosen.
	
	\section{Spacelike CMC surfaces in \(\mathbb{R}^{2,1} \) }
	In this section, we review the DPW method and the construction of spacelike CMC surfaces in \(\mathbb{R}^{2,1 } \) from \(\mathfrak{sl}_2 \mathbb{C} \)-valued 1-forms.
	\subsection{Loop groups}
	First, we define the loop groups, and their loop algebras. Let
	\begin{equation} 
		I = \{\lambda \in \mathbb{C} \ : | \lambda |< 1 \}, \nonumber
	\end{equation}
	\begin{equation}
		{\rm SU}_{1,1} = \left\{g \in {\rm SL}_2 \mathbb{C}\ : \ g \left(\begin{array}{cc}
			1 & 0 \\
			0 & -1
		\end{array} \right)\overline{g}^t = \left(\begin{array}{cc}
			1 & 0 \\
			0 & -1
		\end{array} \right) \right\}, \nonumber
	\end{equation}
	and we define a map \(\sigma : {\rm M}_2 \mathbb{C} \rightarrow {\rm M}_2 \mathbb{C} \) by
	\begin{equation}
		\bf{\sigma}(\gamma) =  
		\left(
		\begin{array}{cc}
			1 & 0 \\ 0 & -1
		\end{array}
		\right)
		\gamma
		\left(
		\begin{array}{cc}
			1 & 0 \\ 0 & -1
		\end{array}
		\right). \nonumber
	\end{equation}
	The loop group of \(\rm SL_2\mathbb{C} \) is
	\begin{equation}
		(\Lambda \rm SL_2 \mathbb{C})_{\sigma} = \{ \gamma = \gamma(\lambda) \in \it C^{\infty}(S^1, \rm SL_2 \mathbb{C}) : \gamma(-\lambda) = \sigma(\gamma(\lambda)) \}. \nonumber
	\end{equation}
	The Lie algebra of the loop group is
	\begin{equation}
		(\Lambda \rm \mathfrak{sl}_2 \mathbb{C})_{\sigma} = \{ A = A(\lambda) \in \it C^{\infty}(S^1, \mathfrak{sl_2} \mathbb{C}):  A(-\lambda) = \sigma(A(\lambda)) \}. \nonumber
	\end{equation}
	We denote other loop groups by
	\begin{equation}
		(\Lambda^+ \rm SL_2 \mathbb{C})_{\sigma } = \left\{ \gamma \in  (\Lambda \rm SL_2 \mathbb{C})_{\sigma}:
		\begin{gathered}
			\gamma \ \text{can be extended holomorphically } \\ \text{to } \  \gamma : I \rightarrow {\rm SL_2 \mathbb{C} }  
		\end{gathered}\right\}.\nonumber
	\end{equation}
	\begin{equation}
		(\Lambda^- \rm SL_2 \mathbb{C})_{\sigma } = \left\{ \gamma \in  (\Lambda {}\rm SL_2 \mathbb{C})_{\sigma}:
		\begin{gathered}
			\gamma \ \text{can be extended holomorphically } \\ \text{to } \  \gamma : \mathbb{C}P^1 - I \rightarrow {\rm SL_2 \mathbb{C} }  
		\end{gathered} \right\}.\nonumber
	\end{equation}
	\begin{equation}
		(\Lambda^{+,>0} \rm SL_2 \mathbb{C})_{\sigma } = \left\{ \gamma \in  (\Lambda^+ \rm SL_2 \mathbb{C})_{\sigma }: \gamma|_{\lambda=0} = \left(
		\begin{array}{cc}
			\rho & 0 \\
			0 & \rho^{-1}
		\end{array}
		\right) \ {\rm for\ some}\ \rho > 0
		\right\}.\nonumber
	\end{equation}
	\begin{equation}
		(\Lambda^{- * } \rm SL_2 \mathbb{C})_{\sigma } = \left\{ \gamma \in  (\Lambda^- \rm SL_2 \mathbb{C})_{\sigma}: \gamma|_{\lambda = \infty } = Id \right\}, \nonumber
	\end{equation}
	\begin{equation}
		(\Lambda {\rm SU_{1,1}})_{\sigma}  = \left\{g \in (\Lambda {\rm SL}_2 \mathbb{C})_{\sigma}\ : \ g(\lambda) \left(\begin{array}{cc}
			1 & 0 \\
			0 & -1
		\end{array} \right)\overline{g(\overline{\lambda^{-1}})}^t = \left(\begin{array}{cc}
			1 & 0 \\
			0 & -1
		\end{array} \right) \right\}, \nonumber
	\end{equation}
	\begin{rem}
		For a complex semisimple Lie group \(G^{\mathbb{C}}\), the loop group \(\Lambda G^{\mathbb{C}}\) is the Banach space of maps from \(S^1\) to \(G^{\mathbb{C}}\) with some \(H^s\)-topology, \(s>1/2\) (see detail in \cite{PS1986}).
	\end{rem}
	In this paper, we use the following splitting theorems:
	\begin{thm}[The Iwasawa factorization for \((\Lambda {\rm SU}_{1,1})_{\sigma }  \), \cite{BD2001}, \cite{BRS2010}]
		There exists an open dense subset \(\mathcal{U} \subset (\Lambda {\rm SL}_2 \mathbb{C} )_{\sigma} \) such that the multiplication 
		\begin{equation}
			((\Lambda {\rm SU}_{1,1})_{\sigma} \sqcup ((\Lambda {\rm SU}_{1,1} )_{\sigma} \cdot w ) ) \times (\Lambda^{+,>0} \rm{SL}_2 \mathbb{C})_{\sigma} \rightarrow \mathcal{U}, \nonumber
		\end{equation}
		is a real-analytic bijective diffeomorphism with respect to the natural smooth manifold structure, where
		\begin{equation}
			w = \left(
			\begin{array}{cc}
				0 & \lambda \\
				-\lambda^{-1} & 0
			\end{array}
			\right). \nonumber
		\end{equation}
		The unique splitting of an element \( \phi \in \mathcal{U} \),
		\begin{equation} \phi = F\cdot B, \nonumber
		\end{equation}
		with \(F \in (\Lambda {\rm SU}_{1,1})_{ \sigma} \sqcup ( (\Lambda {\rm SU}_{1,1} )_{\sigma} \cdot w ) \) and \( B \in (\Lambda^{+,>0} \rm{SL}_2 \mathbb{C})_{\sigma}   \), will be called the Iwasawa factorization.
	\end{thm}
	\begin{thm}[The Birkhoff factorization for \((\Lambda {\rm SL}_2 \mathbb{C} )_{\sigma }  \), \cite{PS1986}]
		There exists an open dense subset \(\mathcal{V} \subset (\Lambda {\rm SL}_2 \mathbb{C} )_{\sigma} \) such that the multiplication 
		\begin{equation}
			(\Lambda^{- * } \rm{SL}_2 \mathbb{C})_{\sigma} \times (\Lambda^+ \rm{SL}_2 \mathbb{C})_{\sigma} \rightarrow \mathcal{V}, \nonumber
		\end{equation}
		is a complex-analytic bijective diffeomorphism. The splitting of an element \( \phi \in \mathcal{V} \),
		\begin{equation} \phi = \phi_- \cdot \phi_+, \nonumber
		\end{equation}
		with \(\phi_- \in (\Lambda^{- * } {\rm SL}_2 \mathbb{C} )_{ \sigma} \) and \( \phi_+ \in (\Lambda^+ \rm{SL}_2 \mathbb{C})_{\sigma}   \), will be called the Birkhoff factorization.
	\end{thm}
	\begin{rem}
		If we replace \((\Lambda^{- * } \rm{SL}_2 \mathbb{C})_{\sigma} \) by \((\Lambda^- \rm{SL}_2 \mathbb{C})_{\sigma} \), then the splitting is not unique, but in this paper we also call this splitting the Birkhoff factorization. A similar splitting of \((\Lambda {\rm GL}_2 \mathbb{C} )_{\sigma } \) can be obtained.
	\end{rem}
	
	\subsection{The DPW method for spacelike CMC- surfaces in \(\mathbb{R}^{2,1} \)}
	Let \(\Sigma \) be a Riemann surface with local coordinate system \(z\) and
	\begin{align}
		\xi = \frac{1}{\lambda} \left(\begin{array}{cc}
			0 & p \\
			q & 0
		\end{array} \right)dz + \sum_{j \ge 0} \xi_j \lambda^j dz \in (\Lambda \frak{sl}_2 \mathbb{C})_{\sigma}, \nonumber
	\end{align}
	where \(p,q \) are holomorphic in \(z \) and \(p \neq 0 \). We call \(\xi \) a holomorphic potential. \vspace{2mm} \\
	In the DPW method, we construct CMC-surfaces from holomorphic potentials as follows: 
	\begin{enumerate}[(1)]
		\item First, we assume that \(\Sigma\) is simply-connected. Given a holomorphic potential \(\xi \), we solve 
		\begin{equation}
			d\phi = \phi \xi, \ \ \phi(z_0)= \phi_0, \label{eqd}
		\end{equation}
		for some base point \(z_0\) and \(\phi_0 \in (\Lambda {\rm SL}_2 \mathbb{C})_{\sigma} \). Since \(\Sigma\) is simply-connected, by the local existence and uniqueness of ordinary differential equations there exists a matrix solution of this equation globally defined on \(\Sigma \). 
		\vspace{2mm}
		\item Then we can split \(\phi \in (\Lambda {\rm SL}_2 \mathbb{C})_{\sigma} \) via Iwasawa factorization into a product
		\begin{equation}
			\phi = F\cdot B\ \ \ {\rm or }\ \ \ F \cdot w \cdot B,\nonumber
		\end{equation}
		on some open neighbourhood \(U \subset \Sigma \) of \(z = z_0 \), where \(F \in (\Lambda {\rm SU }_{1,1})_{\sigma}\), \(B \in (\Lambda^{+,>0} {\rm SL}_2 \mathbb{C})_{\sigma} \). 
		\vspace{2mm}
		\item  Finally, we compute the Sym-Bobenko formula
		\begin{equation}
			f(z,\overline{z},\lambda) = \frac{-i}{2H} \left[\lambda (\partial_\lambda F)F^{-1} + \frac{1}{2}F \left( \begin{array} {cc}
				1 & 0\\
				0 & -1
			\end{array} \right)
			F^{-1} \right], \label{sym}
		\end{equation}
		where \(H \in \mathbb{R}, H \neq 0 \). Then we can see that \(f=f(z,\overline{z},\lambda): U \rightarrow \mathbb{R}^{2,1} \) is a conformal immersion with constant mean curvature \(H \) for all \(\lambda \in S^1 \) (see Theorem 3.8. of \cite{BRS2010}).
	\end{enumerate}
	
	\section{The DPW construction for Smyth potentials near \(z = 0 \) }
	From the argument in section 4 of \cite{GL2014}, a holomorphic potential on \(\mathbb{C}^* \) of the form
	\begin{equation}
		\xi = \frac{1}{\lambda} \left(
		\begin{array}{cc}
			0 & z^{k_0} \\
			z^{k_1} & 0
		\end{array}
		\right)dz, \ \ \ {\rm where}\ k_0, k_1 \in \mathbb{Z}_{ \ge -1}, \nonumber
	\end{equation}
	gives a local solution of the sinh-Gordon equation.
	
	In this paper, we consider the quantum cohomology case \(k_1 = -1 \). In this case, we can not expect \(\phi(0) = Id\). Instead, we use a modified normalization, as in the following proposition. For the case \(k_0 = 3, k_1 = -1 \), we follow section 3 and section 4 of \cite{DGR2010}. The Propositions and Corollary in this section can be proved as in \cite{DGR2010}. The parameters such as \(k_0, k_1\) need to be changed, but the method of proof is exactly the same. For this reason we omit the proofs. Instead of this, we give the proofs in Appendix A,
	\begin{prop}[section 3 of \cite{DGR2010}]\label{prop3.1}
		Let
		\begin{equation}
			L = \left(
			\begin{array}{cc}
				1 & 0 \\
				\lambda^{-1} \log{\frac{z}{2}} & 1
			\end{array}
			\right)
			\left(
			\begin{array}{cc}
				y_0 & \lambda z (y_0)_z \\
				\frac{1}{4} y_1 & y_0 + \lambda z \frac{1}{4} (y_1)_z
			\end{array}
			\right), \nonumber
		\end{equation}
		where \(y_0:\mathbb{C}^* \rightarrow \mathbb{C}\) and \(y_1:\mathbb{C}^* \rightarrow \mathbb{C}\) are given by
		\begin{equation}
			y_0(z,\lambda) = \sum_{j \ge 0} \frac{\lambda^{-2j} z^{4j} }{16^j(j!)^2} ,\ \ y_1(z,\lambda) = -2 \lambda^{-1} \sum_{j \ge 1} (1 + \cdots + \frac{1}{j} ) \frac{\lambda^{-2j} z^{4j} }{16^j (j!)^2}. \nonumber  
		\end{equation}
		Then \(L \) is a solution of
		\begin{equation}
			L^{-1} dL = \frac{1}{\lambda} \left(
			\begin{array}{cc}
				0 & z^3 \\
				z^{-1} & 0
			\end{array}
			\right)dz, \nonumber
		\end{equation}
		on \(\mathbb{C}^* \). Here, we regard \(\log{z}\) as a multi-valued function on \(\mathbb{C}^*\).
	\end{prop}
	The solution \(L \) in Proposition 3.1 is a multi-valued function on \(\mathbb{C}^* \). We can split \(\phi \) near \(z = 0 \) by the Iwasawa factorization.
	\begin{prop}[Theorem 4.3 of \cite{DGR2010}]
		For any \(a \in \mathbb{R}_{>0} \) and \(y_0, y_1 \) as in Proposition 3.1,
		\begin{equation}
			\phi = \left(
			\begin{array}{cc}
				\sqrt{-a} & -\lambda / \sqrt{-a} \\
				0 & 1 / \sqrt{-a}
			\end{array}
			\right)
			\left(
			\begin{array}{cc}
				1 & 0 \\
				\lambda^{-1} \log{\frac{z}{2}} & 1
			\end{array}
			\right) \left(
			\begin{array}{cc}
				y_0 & \lambda z (y_0)_z \\
				\frac{1}{4} y_1 & y_0 + \lambda z \frac{1}{4} (y_1)_z
			\end{array}
			\right), \nonumber
		\end{equation}
		admits an Iwasawa factorization
		\begin{equation*}
			\phi = F \left(
			\begin{array}{cc}
				0 & \lambda \\
				-\lambda^{-1} & 0
			\end{array}
			\right) B,\ \ \ F \in (\Lambda {\rm SU}_{1,1})_{\sigma},\ B \in (\Lambda^{+,>0} {\rm SL}_2 \mathbb{C})_{\sigma}, \nonumber
		\end{equation*}
		on some domain \(U = V \cap \mathbb{C} \backslash  (-\infty,0] \), where \(V \) is a neighbourhood of \(z = 0 \) in \(\mathbb{C} \).
	\end{prop}
	From the Iwasawa factorization of \(\phi \), we obtain a local solution of the sinh-Gordon equation as follows.
	\begin{prop}[Lemma 4.2 of \cite{DGR2010}]
		For any \(e^{i\theta } \in S^1 \),
		\begin{equation}
			\phi(z e^{i\theta}, \lambda e^{i2\theta} ) = \delta T^{-1} \phi(z,\lambda) T, \nonumber
		\end{equation}
		where \(\delta = \delta(\theta) \in (\Lambda {\rm SU}_{1,1 } )_{\sigma } \) and
		\begin{equation}
			T = T(\theta) = \left( \begin{array}{cc}
				e^{-i\theta } & 0 \\
				0 & e^{i\theta }
			\end{array} \right) \in U(1). \nonumber
		\end{equation} 
	\end{prop}
	\begin{cor}[Proposition 5.2 of \cite{DGR2010}]\label{cor3.4}
		Let
		\begin{equation}
			\phi = F \left(
			\begin{array}{cc}
				0 & \lambda \\
				-\lambda^{-1} & 0
			\end{array}
			\right) B, \nonumber
		\end{equation}
		be an Iwasawa factorization on some domain \(U = V \cap \mathbb{C} \backslash  (-\infty,0] \), where \(V \) is a neighbourhood of \(z = 0 \) in \(\mathbb{C} \). We put
		\begin{equation}
			B(z,\lambda = 0) = \left(
			\begin{array}{cc}
				e^{\frac{v}{2} } & 0 \\
				0 & e^{- \frac{v}{2} }
			\end{array}
			\right), \nonumber
		\end{equation}
		then \(v(z, \overline{z}) = v(|z|) \) and \(v:U \rightarrow \mathbb{R} \) satisfies
		\begin{equation}
			\frac{1}{4} \left(v_{rr} + \frac{1}{r} v_r  \right) - r^6 e^{2v} + r^{-2} e^{-2v} = 0, \nonumber
		\end{equation}
		and \(e^{\frac{v}{2}} \sim \sqrt{-a -2 \log{r} } \) as \(r \rightarrow 0\), where \(r = |z| \).
	\end{cor}
	For \(v \) in Corollary 3.4, we put \(e^u = r^2 e^v \) and \(x = \frac{r^2}{2} \), then \(u \) satisfies
	\begin{equation}
		\frac{1}{4} \left(u_{xx} + \frac{1}{x} u_x \right) = e^{2u} - e^{-2u}, \nonumber
	\end{equation} 
	and \(u \sim (1 + o(x) )\log{x} \) as \(x \rightarrow 0 \). Thus, we obtain a local solution \(u:U \rightarrow \mathbb{R} \) of the sinh-Gordon equation near \(z = 0 \). \vspace{3mm} \\
	\ \ In this paper, our goal is to obtain the global Iwasawa factorization of \(\phi \) and give a global solution of the sinh-Gordon equation. For the canonical solution \(L \), we want to find a suitable dressing action \(\gamma_0 \) such that the Iwasawa factorization of \(\gamma_0 L \) is globally defined.
	
	\section{The asymptotic behaviour of \(\phi \) at large \(z \) }
	In section 3, we consider the DPW construction near \(z = 0 \) and obtain the local solution of the sinh-Gordon equation. On the other hand, our goal is to obtain the global Iwasawa factorization. Thus, we need to consider the DPW construction at large \(z \). In this section, for \(a > 0 \) we investigate the asymptotic behaviour of 
	\begin{equation}
		\phi = \left(
		\begin{array}{cc}
			\sqrt{-a} & -\lambda / \sqrt{-a} \\
			0 & 1 / \sqrt{-a} 
		\end{array}
		\right) \left(
		\begin{array}{cc}
			1 & 0 \\
			\lambda^{-1} \log{\frac{z}{2}} & 1
		\end{array}
		\right)\left(
		\begin{array}{cc}
			y_0 & \lambda z (y_1)_z \\
			\frac{1}{4}y_1 & y_0 + \lambda z \frac{1}{4} (y_1)_z
		\end{array}
		\right), \nonumber
	\end{equation}
	at large \(z \) (more precisely, we investigate the behaviour at \(\lambda^{-1} z^2 = \infty \)). \\
	We know (section 3 of \cite{DGR2010}) that the scalar equation corresponding to the Smyth potential with \(k_0 = 3,\ k_1 = -1 \) is
	\begin{equation}
		y_{zz } + \frac{1}{z} y_z - \lambda^{-2 } z^2 y = 0. \nonumber
	\end{equation}
	Fix \(\lambda \) and by changing of the coordinate \(x = \lambda^{-1} \frac{z^2}{2} \), this equation can be transformed to the modified Bessel equation (see definition B in Appendix or \cite{WW1996}) with order zero
	\begin{equation}
		y_{xx } + \frac{1}{x} y_x - y = 0. \nonumber
	\end{equation}
	The properties of the Bessel functions are well-known and we shall use them to investigate the properties of \(\phi \). First, we consider the universal covering \(\tilde{\lambda} \) of \(\lambda \in \mathbb{C}^* \) to investigate the multi-valued function \(\phi \) and express \(\phi \) in terms of the Bessel functions:
	\begin{prop}\label{prop4.1}
		Let \(Pr: \mathbb{C} \rightarrow \mathbb{C}^*: \tilde{\lambda } \mapsto e^{\tilde{\lambda } } = \lambda \) be the universal covering of \(\mathbb{C}^*  \), \(I_0(x)\) the modified Bessel function of the first kind and \(Y_0(x) \) the Bessel function of the second kind (the definitions are given in Appendix B).\\
		For \(z \in \mathbb{R}_{>0}\), we have
		\begin{align}
			\phi(z,\lambda ) &= \left(
			\begin{array}{cc}
				1+p & \lambda p \\
				-\lambda^{-1} p & 1-p
			\end{array}
			\right)
			\left(
			\begin{array}{cc}
				\sqrt{-a} & -\lambda / \sqrt{-a} \\
				0 & 1 / \sqrt{-a}
			\end{array}
			\right) \cdot \nonumber \\
			&\ \ \ \ \ \left(
			\begin{array}{cc}
				1 & 0 \\
				- \lambda^{-1} \left( \frac{\gamma}{2} + i \frac{\pi}{4} \right) & \frac{\pi}{4}
			\end{array}
			\right) \left(
			\begin{array}{cc}
				I_0(\lambda^{-1} \frac{z^2}{2}) & \lambda z (I_0(\lambda^{-1} \frac{z^2}{2}))_z \\
				\lambda^{-1} Y_0(i\lambda^{-1} \frac{z^2}{2}) & z(Y_0(i\lambda^{-1} \frac{z^2}{2} )_z )
			\end{array}
			\right),  \nonumber
		\end{align}
		where \( p = \frac{1}{2a} \log{\lambda} = \frac{\tilde{\lambda}}{2a}\), \(\gamma \) is the Euler constant. For fixed \(z\), \(p\) and \(Y_0\) are multi-valued functions in \(\lambda \in \mathbb{C}^*\) but we can regard \(p(\lambda ), Y_0(i\lambda^{-1} \frac{z^2}{2} ) \) as single-valued functions in \(\tilde{\lambda } \in \mathbb{C}\).
	\end{prop}
	\begin{proof}
		From Definition B.2 and Proposition B.1 in Appendix B, the expansion series of \(I_0(x),\ Y_0(ix) \) at \(x = 0 \) are given by
		\begin{equation}
			I_0(x) = \sum_{j \ge 0} \frac{1}{(j!)^2} \left(\frac{x}{2} \right)^{2j}, \nonumber
		\end{equation}
		and 
		\begin{equation}
			Y_0(ix) = \frac{2}{\pi} \left\{ I_0(x) \log{\frac{x}{2} } + \left( \gamma + i \frac{\pi}{2} \right) I_0(x) - \sum_{j \ge 1} \frac{1}{(j!)^2} (1+ \cdots + \frac{1}{j} ) \left(\frac{x}{2} \right)^{2j }  \right\}. \nonumber
		\end{equation}
		From the expansion series of \(y_0,\ y_1 \), we have \(y_0(z) = I_0(\lambda^{-1} \frac{z^2}{2} ) \) 
		and
		\begin{align}
			y_1(z) &= -2\lambda^{-1} \left\{ I_0\left(\lambda^{-1} \frac{z^2}{2} \right) \log{\left(\frac{\lambda^{-1} z^2 }{4} \right) } \right. \nonumber \\
			&\hspace{3cm} \left. + \left( \gamma + i\frac{\pi}{2} \right) I_0\left(\lambda^{-1} \frac{z^2}{2} \right) - \frac{\pi}{2} Y_0\left(i \lambda^{-1} \frac{z^2}{2} \right)  \right\}, \nonumber
		\end{align}
		thus, we obtain the stated result. \qed
	\end{proof}
	\begin{rem}
		We remark that \(\phi  \) is a holomorphic function of two variables \(z, \lambda \), but \(I_0, Y_0 \) can be regarded as a holomorphic function of one variable \(x = \lambda^{-1 } \frac{z^2}{2} \). We can see that \(x \) satisfies the homogeneity condition i.e. \(x \) is invariant under the transformation \(z \mapsto e^{i\theta }z,\ \lambda \mapsto e^{i2\theta }\lambda \) for all \(\theta \in \mathbb{R} \). Thus, \(I_0, Y_0 \) also satisfy the homogeneity condition.
	\end{rem}
	In order to investigate the behaviour of \(\phi \) at large \(\lambda^{-1 } z^2 \), we use the asymptotic expansion of the Bessel functions:
	\begin{prop}[17-5, 17-6 of \cite{WW1996}]
		For \(- \frac{3}{2}\pi < {\rm arg}(x) < \frac{\pi}{2}\),
		\begin{equation}
			I_0(x) = \sqrt{\frac{2 }{\pi x } } e^{- i\frac{\pi}{4} } \left\{ \left(1 + T_1(x) \right) \cos{\left(ix - \frac{\pi}{4} \right)} - T_2(x) \sin{\left(ix - \frac{\pi}{4} \right) } \right\} , \nonumber
		\end{equation}
		\begin{equation}
			Y_0(ix) = \sqrt{\frac{2 }{\pi x } } e^{- i\frac{ \pi}{4} } \left\{ \left(1 + T_1(x) \right) \sin{\left(ix - \frac{\pi}{4} \right)} + T_2(x) \cos{\left(ix - \frac{\pi}{4} \right) }  \right\} , \nonumber
		\end{equation}
		where \(- \frac{3}{4}\pi < {\rm arg}\left( \sqrt{x} \right) < \frac{\pi}{4}\) and 
		\begin{equation}
			|T_1(x)| \le  C_1 |x|^{-2},\ \ |T_2(x)| \le C_2 |x|^{-1},\ \ \  C_1,\ C_2:{\rm constant}. \nonumber 
		\end{equation}
	\end{prop}
	\begin{rem}
		We remark that \(T_1(x),\ T_2(x) \) are holomorphic function on \(- \frac{3}{2} \pi < {\rm arg}(x) < \frac{\pi}{2} \) and \(T_1,T_2 \) converge to \(0 \) as \(x \rightarrow \infty \) on this domain. In Proposition B.3, we choose \(\alpha=0\) and we denote \(T_1(x)=T_1^0(ix), T_2(x)=T_2^0(ix)\). (See also 17-5, 17-6 of \cite{WW1996}. They use the symbol \(U_{\alpha}, V_{\alpha}\) instead of \(T_1^{\alpha}, T_2^{\alpha}\). More precisely, \(U_{\alpha}, V_{\alpha}\) are the asymptotic expansion of \(T_1^{\alpha}, T_2^{\alpha}\) respectively.)
	\end{rem}
	In this section, \(T_1,\ T_2 \rightarrow 0 \) as \(\lambda^{-1} z^2 \rightarrow \infty \) on \(-\frac{3}{2}\pi < {\rm arg}\left(\lambda^{-1} z^2 \right) < \frac{\pi}{2} \). From Proposition 4.1, 4.2, we split \(\phi \) into a \(z \)-independent part, a divergent part at \(\lambda^{-1 } z^2 = \infty \), a convergent part at \(\lambda^{-1 } z^2 = \infty \) and a \(\lambda \)-independent diagonal matrix. This splitting gives the asymptotic behaviour of \(\phi \) at large \(\lambda^{-1}z^2 \) with the leading term
	\begin{equation}
		\exp{\left(\lambda^{-1} \frac{z^2}{2} \left( \begin{array}{cc}
				0 & 1 \\
				1 & 0
			\end{array} \right) \right) }. \nonumber
	\end{equation}
	Since \(\phi \) is a multi-valued function for \(\lambda \), we consider the universal covering of \(\mathbb{C}^* \) as in Proposition 4.1.
	\begin{prop}
		Let \(Pr: \mathbb{C} \rightarrow \mathbb{C}^*: \tilde{\lambda } \mapsto e^{\tilde{\lambda } } = \lambda \) be the universal covering of \(\mathbb{C}^*  \). We define
		\begin{equation}
			\tilde{\Omega}_0  = \left\{\tilde{\lambda } \in \mathbb{C} \ | \ - \frac{\pi}{2} < {\rm Im}(\tilde{\lambda }) < \frac{3}{2}\pi  \right\}, \nonumber
		\end{equation}
		and \(\tilde{\Omega}_m  = \tilde{\Omega}_0 - im\pi \) for \(m \in \mathbb{Z} \). \\
		For \(z \in \mathbb{R}_{>0},\ \lambda = e^{\tilde{\lambda } }\ (\tilde{\lambda } \in \tilde{\Omega }_m ) \), we split	
		\begin{equation}
			\phi(z,\lambda) = H(\lambda) A_m(\lambda) \phi_0(z,\lambda) K_m(z,\lambda) C(z), \nonumber
		\end{equation}
		where
		\begin{enumerate}
			\item [(a)] \(H = \left(
			\begin{array}{cc}
				1+p & \lambda p \\
				-\lambda^{-1} p & 1-p
			\end{array}
			\right) \), where \(p = \frac{1}{2a} \log{\lambda} = \frac{\tilde{\lambda}}{2a},\ \tilde{\lambda} \in \tilde{\Omega}_m\), \vspace{3mm}

			\item [(b)] \(A_m =  \frac{i}{2 \sqrt{a \pi} } \left(
			\begin{array}{cc}
				i^m \sqrt{\lambda} \left\{ q \gamma e^{-i \frac{\pi}{4}} + s \right\}
				& i^{-m} \sqrt{\lambda} \left\{q \gamma e^{i\frac{\pi}{4} } - \overline{s} \right\}
				\\ i^m \sqrt{\lambda}^{-1} \left\{ \sqrt{2 } \gamma e^{-i \frac{\pi}{4}} - s \right\}
				& i^{-m} \sqrt{\lambda}^{-1} \left\{ \sqrt{2 } \gamma e^{i \frac{\pi}{4}} + \overline{s} \right\}
			\end{array} \right) \), \\
			where
			\begin{equation}
				q = 2 \left(2\sqrt{\frac{a }{\gamma} } - \sqrt{\frac{\gamma }{a } } \right) \sqrt{\frac{a}{2\gamma} } ,\ \ \ s = \left(\sqrt{2} m e^{i\frac{\pi}{4} } - 1\right)\pi, \nonumber
			\end{equation} \vspace{3mm}
			
			\item [(c)] \(\phi_0(z,\lambda) = \exp{\left( \lambda^{-1} \frac{z^2}{2}
				\left(
				\begin{array}{cc}
					0 & 1 \\
					1 & 0
				\end{array} \right)
				\right) }, \) \vspace{3mm}
			
			\item [(d)] \(K_m = \left(\begin{array}{cc}
				\tilde{T}_1(z,\lambda) & \tilde{T}_2(z,\lambda) + \lambda\left(z^{-1} \tilde{T}_2(z,\lambda) \right)_z \\
				\tilde{T}_2(z,\lambda) & \tilde{T}_1(z,\lambda) + \lambda \left(z^{-2}\tilde{T}_2(z,\lambda) \right)_z
			\end{array} \right) \vspace{2mm} \\ \hspace{5mm} = (Id + o(\lambda^{-1 } z^2 ) ) \)\ \ \ \ \ as\ \ \(\lambda^{-1}z^2 \rightarrow \infty \) for \(\tilde{\lambda} \in \tilde{\Omega}_m\),\\
			where
			\begin{align}
				&\tilde{T}_1(z,\lambda) = 1 + T_1\left(\lambda^{-1} \frac{z^2}{2}e^{-im\pi} \right), \nonumber\\
				&\tilde{T}_2(z,\lambda) = (-1)^{m+1}iT_2\left(\lambda^{-1} \frac{z^2}{2}e^{-im\pi} \right),\ \ \ \ \ \tilde{\lambda} \in \tilde{\Omega}_m, \nonumber
			\end{align}
			\item [(e)] \(C = \left(
			\begin{array}{cc}
				z^{-1 } & 0 \\
				0 & z
			\end{array}
			\right) \).
			
		\end{enumerate}
	\end{prop}
	\begin{proof}
		From Definition B2 and Proposition B1, the continuation formulas of the Bessel functions are given by, for all \(m \in \mathbb{Z} \),
		\begin{equation}
			\left(
			\begin{array}{c}
				I_0(xe^{im\pi }) \\ Y_0(ixe^{im\pi })
			\end{array}
			\right) = \left(
			\begin{array}{cc}
				1 & 0 \\
				2im & 1
			\end{array}
			\right) \left(
			\begin{array}{c}
				I_0(x) \\ Y_0(ix)
			\end{array}
			\right). \nonumber
		\end{equation}
		We use the form of \(\phi \) in Proposition 4.1. From Proposition 4.1 and the continuation formulas above, for \(-\frac{3}{2} \pi + m\pi < {\rm arg}(\lambda^{-1} z^2) < \frac{\pi}{2} + m\pi \), i.e. \(-\frac{3}{2} \pi < {\rm arg}(\lambda^{-1} z^2 e^{-im\pi }) < \frac{\pi}{2} \), we have
		\begin{align}
			&\phi(z,\lambda) \nonumber \\ &= \left(
			\begin{array}{cc}
				1+p(\lambda) & \lambda p(\lambda) \\
				-\lambda^{-1} p(\lambda) & 1-p(\lambda)
			\end{array}
			\right) \left(
			\begin{array}{cc}
				\sqrt{-a} & -\lambda / \sqrt{-a} \\
				0 & 1 / \sqrt{-a}
			\end{array}
			\right) \left(
			\begin{array}{cc}
				1 & 0 \\
				- \lambda^{-1} \left( \frac{\gamma}{2} + i \frac{\pi}{4} \right) & \frac{\pi}{4}
			\end{array}
			\right) \cdot \nonumber \\
			&\ \ \ \ \ \ \ \ \left(
			\begin{array}{cc}
				I_0(\lambda^{-1} \frac{z^2}{2} e^{-im\pi } \cdot e^{im\pi } ) & \lambda z (I_0(\lambda^{-1} \frac{z^2}{2} e^{-im\pi } \cdot e^{im\pi } ))_z \\
				\lambda^{-1} Y_0(i \lambda^{-1} \frac{z^2}{2} e^{-im\pi } \cdot e^{im\pi } ) & z(Y_0(i \lambda^{-1} \frac{z^2}{2} e^{-im\pi } \cdot e^{im\pi } ) )_z 
			\end{array}
			\right)  \nonumber \\ 	
			&= \left(
			\begin{array}{cc}
				1+p(\lambda) & \lambda p(\lambda) \\
				-\lambda^{-1} p(\lambda) & 1-p(\lambda)
			\end{array}
			\right) \left( \begin{array}{cc}
				1 & 0 \\
				0 & \lambda^{-1 }
			\end{array} \right) \left( \begin{array}{cc}
				\sqrt{-a } & -1/ \sqrt{-a } \\
				0 & 1 / \sqrt{-a }   
			\end{array} \right) \nonumber \\ 
			&\hspace{1cm}  \left(
			\begin{array}{cc}
				1 & 0 \\
				- \left( \frac{\gamma}{2} + i \frac{\pi}{4} \right) & \frac{\pi}{4}
			\end{array}
			\right) \left(
			\begin{array}{cc}
				1 & 0 \\
				2im & 1
			\end{array}
			\right) \nonumber \\
			&\hspace{2.5cm} \left(
			\begin{array}{cc}
				I_0(\lambda^{-1} \frac{z^2}{2} e^{-im\pi } ) & z (I_0(\lambda^{-1} \frac{z^2}{2} e^{-im\pi } ))_z \\
				Y_0(i \lambda^{-1} \frac{z^2}{2} e^{-im\pi } ) & z(Y_0(i \lambda^{-1} \frac{z^2}{2} e^{-im\pi } ) )_z 
			\end{array}
			\right) \left( \begin{array}{cc}
				1 & 0 \\
				0 & \lambda
			\end{array} \right).  \nonumber
		\end{align}
		We simplify the forms of \(I_0,Y_0 \) in Proposition 4.2 as a column vector, then for \(- \frac{3}{2} \pi < {\rm arg}(\lambda^{-1} \frac{z^2}{2} e^{-im\pi } ) < \frac{\pi}{2}\)
		\begin{align}
			&\left(
			\begin{array}{c}
				I_0(\lambda^{-1}\frac{z^2}{2} e^{-im\pi } ) ) \\ Y_0(\lambda^{-1}i\frac{z^2}{2} e^{-im\pi } ) )
			\end{array}
			\right) \nonumber \\ 
			&= 2 \sqrt{\frac{\lambda}{\pi} } e^{- i\frac{ \pi}{4} } i^m z^{-1} \left(
			\begin{array}{cc}
				\cos{\left((-1)^m i\lambda^{-1}\frac{z^2}{2} - \frac{\pi}{4} \right)} & - \sin{\left((-1)^m i\lambda^{-1}\frac{z^2}{2} - \frac{\pi}{4} \right)} \\
				\sin{\left((-1)^m i\lambda^{-1}\frac{z^2}{2} - \frac{\pi}{4} \right)} &
				\cos{\left((-1)^m i\lambda^{-1}\frac{z^2}{2} - \frac{\pi}{4} \right)}
			\end{array}
			\right) \nonumber \\
			&\hspace{8cm} \cdot \left(
			\begin{array}{c}
				1 + T_1 \left(
				\lambda^{-1 } \frac{z^2}{2} e^{-im\pi }
				\right) \\ T_2 \left(
				\lambda^{-1 } \frac{z^2}{2} e^{-im\pi }
				\right)
			\end{array}
			\right) \nonumber \\
			&= \sqrt{\frac{2 \lambda}{\pi} } \cdot i^m \left(
			\begin{array}{cc}
				e^{ - i\frac{\pi}{4}} & e^{ i\frac{\pi}{4}} \\
				e^{ i\frac{3}{4}\pi} & e^{ i\frac{\pi}{4}}
			\end{array}
			\right) \exp{\left((-1)^m \lambda^{-1} \frac{z^2}{2}
				\left(
				\begin{array}{cc}
					0 & 1 \\
					1 & 0
				\end{array} \right)
				\right) } \nonumber \\
			&\hspace{7cm} \cdot \left(
			\begin{array}{c}
				z^{-1 } + z^{-1} T_1 \left(
				\lambda^{-1 } \frac{z^2}{2} e^{-im\pi }
				\right) \\ -i z^{-1} T_2 \left(
				\lambda^{-1 } \frac{z^2}{2} e^{-im\pi }
				\right)
			\end{array}
			\right) \nonumber \\
			&= \sqrt{\frac{2 \lambda}{\pi} } \left(
			\begin{array}{cc}
				e^{ - i\frac{\pi}{4}} & e^{ i\frac{\pi}{4}} \\
				e^{ i\frac{3}{4}\pi} & e^{ i\frac{\pi}{4}}
			\end{array}
			\right) \left( \begin{array}{cc}
				i^m & 0 \\
				0 & i^{-m } 
			\end{array} \right) \exp{\left(\lambda^{-1} \frac{z^2}{2}
				\left(
				\begin{array}{cc}
					0 & 1 \\
					1 & 0
				\end{array} \right)
				\right) } \nonumber \\
			&\hspace{6cm} \cdot \left(
			\begin{array}{c}
				z^{-1 } + z^{-1} T_1 \left(
				\lambda^{-1 } \frac{z^2}{2} e^{-im\pi }
				\right) \\ 
				(-1)^{m+1} i z^{-1} T_2 \left(
				\lambda^{-1 } \frac{z^2}{2} e^{-im\pi }
				\right)
			\end{array}
			\right) \nonumber
		\end{align}
		and
		\begin{align}
			&\left(
			\begin{array}{c}
				z(I_0(\lambda^{-1} \frac{z^2}{2} e^{-im\pi } ) )_z \\ z(Y_0(\lambda^{-1} i \frac{z^2}{2} e^{-im\pi} ) )_z
			\end{array}
			\right) \nonumber \\ 
			&= \sqrt{\frac{2 \lambda}{\pi} } \left(
			\begin{array}{cc}
				e^{ - i\frac{\pi}{4}} & e^{ i\frac{\pi}{4}} \\
				e^{ i\frac{3}{4}\pi} & e^{ i\frac{\pi}{4}}
			\end{array}
			\right) \left( \begin{array}{cc}
				i^m & 0 \\
				0 & i^{-m } 
			\end{array} \right) \exp{\left(\lambda^{-1} \frac{z^2}{2}
				\left(
				\begin{array}{cc}
					0 & 1 \\
					1 & 0
				\end{array} \right)
				\right) } \nonumber \\
			&\hspace{2mm} \cdot \left(
			\begin{array}{c}
				(-1)^{m+1} i \lambda^{-1 } z T_2 \left(
				\lambda^{-1 } \frac{z^2}{2} e^{-im\pi }
				\right) + z \left(z^{-1 } + z^{-1} T_1 \left(
				\lambda^{-1 } \frac{z^2}{2} e^{-im\pi }
				\right) \right)_z \\ 
				\lambda^{-1 } z \left( 1 + T_1 \left(
				\lambda^{-1 } \frac{z^2}{2} e^{-im\pi }
				\right) \right) + (-1)^{m+1} i z \left(z^{-1} T_2 \left(
				\lambda^{-1 } \frac{z^2}{2} e^{-im\pi }
				\right) \right)_z
			\end{array}
			\right). \nonumber 
		\end{align}
		Thus, we obtain
		\begin{align}
			&\phi(z,\lambda) \nonumber \\
			&= \left(
			\begin{array}{cc}
				1+p(\lambda) & \lambda p(\lambda) \\
				-\lambda^{-1} p(\lambda) & 1-p(\lambda)
			\end{array}
			\right) \cdot \sqrt{\frac{2\lambda }{\pi } } \left( \begin{array}{cc}
				1 & 0 \\
				0 & \lambda^{-1 }
			\end{array} \right) \left(
			\begin{array}{cc}
				\sqrt{-a} & -1 / \sqrt{-a} \\
				0 & 1 / \sqrt{-a}
			\end{array}
			\right) \nonumber \\
			&\hspace{5mm} \left(
			\begin{array}{cc}
				1 & 0 \\
				- \left( \frac{\gamma}{2} + i \frac{\pi}{4} \right) & \frac{\pi}{4}
			\end{array}
			\right) \left( \begin{array}{cc}
				i^m & 0 \\
				0 & i^{-m } 
			\end{array} \right)
			\left(
			\begin{array}{cc}
				1 & 0 \\
				2im & 1
			\end{array}
			\right) \left(
			\begin{array}{cc}
				e^{ - i\frac{\pi}{4}} & e^{ i\frac{\pi}{4}} \\
				e^{ i\frac{3}{4}\pi} & e^{ i\frac{\pi}{4}}
			\end{array}
			\right) \nonumber \\
			&\hspace{3cm} \cdot \exp{\left(\lambda^{-1} \frac{z^2}{2}
				\left(
				\begin{array}{cc}
					0 & 1 \\
					1 & 0
				\end{array} \right)
				\right) } \nonumber \\
			&\hspace{3mm} \cdot \left( \begin{array}{c}
				z^{-1 } + z^{-1} T_1 \left(
				\lambda^{-1 } \frac{z^2}{2} e^{-im\pi }
				\right) \\
				(-1)^{m+1} i z^{-1} T_2 \left(
				\lambda^{-1 } \frac{z^2}{2} e^{-im\pi }
				\right)
			\end{array} \right. \nonumber \\
			&\hspace{5mm} \left. \begin{array}{c}
				(-1)^{m+1} i \lambda^{-1 } z T_2 \left(
				\lambda^{-1 } \frac{z^2}{2} e^{-im\pi }
				\right) + z \left(z^{-1 } + z^{-1} T_1 \left(
				\lambda^{-1 } \frac{z^2}{2} e^{-im\pi }
				\right) \right)_z \\ 
				\lambda^{-1 } z \left( 1 + T_1 \left(
				\lambda^{-1 } \frac{z^2}{2} e^{-im\pi }
				\right) \right) + (-1)^{m+1} i z \left(z^{-1} T_2 \left(
				\lambda^{-1 } \frac{z^2}{2} e^{-im\pi }
				\right) \right)_z
			\end{array} \right) \nonumber \\
			&\hspace{8cm} \left( \begin{array}{cc}
				1 & 0 \\
				0 & \lambda
			\end{array} \right) \nonumber
		\end{align}
		We have
		\begin{align}
			&\sqrt{\frac{2\lambda }{\pi } } \left( \begin{array}{cc}
				1 & 0 \\
				0 & \lambda^{-1 }
			\end{array} \right) \left(
			\begin{array}{cc}
				\sqrt{-a} & -1 / \sqrt{-a} \\
				0 & 1 / \sqrt{-a}
			\end{array}
			\right) \left(
			\begin{array}{cc}
				1 & 0 \\
				- \left( \frac{\gamma}{2} + i \frac{\pi}{4} \right) & \frac{\pi}{4}
			\end{array}
			\right) \nonumber \\
			&\hspace{2cm} \left( \begin{array}{cc}
				i^m & 0 \\
				0 & i^{-m}
			\end{array} \right) \left( \begin{array}{cc}
				1 & 0 \\
				2im & 1
			\end{array} \right) \left(
			\begin{array}{cc}
				e^{ - i\frac{\pi}{4}} & e^{ i\frac{\pi}{4}} \\
				e^{ i\frac{3}{4}\pi} & e^{ i\frac{\pi}{4}}
			\end{array}
			\right) = A_m(\lambda), \nonumber
		\end{align}
		\begin{align}
			&\left( \begin{array}{c}
				z^{-1 } + z^{-1} T_1 \left(
				\lambda^{-1 } \frac{z^2}{2} e^{-im\pi }
				\right) \\
				(-1)^{m+1} i z^{-1} T_2 \left(
				\lambda^{-1 } \frac{z^2}{2} e^{-im\pi }
				\right)
			\end{array} \right. \nonumber \\
			&\hspace{3mm} \left. \begin{array}{c}
				(-1)^{m+1} i \lambda^{-1 } z T_2 \left(
				\lambda^{-1 } \frac{z^2}{2} e^{-im\pi }
				\right) + z \left(z^{-1 } + z^{-1} T_1 \left(
				\lambda^{-1 } \frac{z^2}{2} e^{-im\pi }
				\right) \right)_z \\ 
				\lambda^{-1 } z \left(1 + T_1 \left(
				\lambda^{-1 } \frac{z^2}{2} e^{-im\pi }
				\right) \right) + (-1)^{m+1} i z \left(z^{-1} T_2 \left(
				\lambda^{-1 } \frac{z^2}{2} e^{-im\pi }
				\right) \right)_z
			\end{array} \right) \nonumber \\
			&\hspace{5cm} \left( \begin{array}{cc}
				1 & 0 \\
				0 & \lambda
			\end{array} \right) = K_m(z,\lambda) C(z), \nonumber
		\end{align}
		and then we obtain
		\begin{equation}
			\phi(z,\lambda) =  H(\lambda) A_m(\lambda) \phi_0(z,\lambda) K_m(z,\lambda) C(z). \nonumber
		\end{equation}
		\qed
	\end{proof}
	\begin{rem}
		We remark that \(H,\ A_m \) are \(z \)-independent and \(H\) satisfies
		\begin{equation}
			\overline{H(\overline{\lambda}^{-1})}^t \left(\begin{array}{cc}
				1 & 0 \\
				0 & -1
			\end{array}\right)H(\lambda) =\left(\begin{array}{cc}
				1 & 0 \\
				0 & -1
			\end{array}\right), \nonumber
		\end{equation}
		on \(\tilde{\Omega}_m \). But
		\begin{equation}
			\overline{A(\overline{\lambda}^{-1})}^t \left(\begin{array}{cc}
				1 & 0 \\
				0 & -1
			\end{array}\right)A(\lambda) =\left(\begin{array}{cc}
				-1 & 0 \\
				0 & 1
			\end{array}\right), \nonumber
		\end{equation}
		if and only if \(a = \gamma \). Obviously, \(\phi_0 \) is divergent at \(\lambda^{-1 } z^2 = \infty \). On the other hand, \(K_m \) is holomorphic in \(\tilde{\lambda } \) on \(\tilde{\Omega}_m \) and converges to \(Id \) at \(\lambda^{-1 } z^2 = 0 \) on \(\tilde{\Omega}_m \).
	\end{rem}
	From the asymptotic behaviour of \(\phi \), in next section we investigate the behaviour of the Iwasawa factorization of \(\phi \) at \(\lambda^{-1} z^2  = \infty \).
	
	\section{Global Iwasawa factorization}
	In section 3, we proved that the holomorphic potential
	\begin{equation}
		\xi = \frac{1}\lambda \left(
		\begin{array}{cc}
			0 & z^3 \\
			z^{-1 } & 0
		\end{array}
		\right)dz, \nonumber
	\end{equation}
	gives a local solution of the sinh-Gordon equation near \(z=0 \) (see Corollary 3.4) by using the Iwasawa factorization
	\begin{equation}
		\phi = \left(
		\begin{array}{cc}
			\sqrt{-a} & -\lambda / \sqrt{-a} \\
			0 & 1 / \sqrt{-a} 
		\end{array}
		\right) L = F \left( \begin{array}{cc}
			0 & \lambda \\
			-\lambda^{-1 } & 0
		\end{array} \right) B. \nonumber
	\end{equation}
	
	\subsection{The Birkhoff factorization from the Iwasawa factorization}
	In this section, we want to find a suitable \(a > 0 \) such that the Iwasawa factorization is global on \(\mathbb{C} \backslash  (-\infty,0] \). To investigate this problem, we consider the following well-known relation with the Birkhoff factorization:
	\begin{prop}
		Fix \(z \in \mathbb{C} \backslash (-\infty,0] \). Let
		\begin{equation}
			g = \overline{\phi(z,\overline{\lambda}^{-1} ) }^t \left(
			\begin{array}{cc}
				1 & 0 \\
				0 & -1
			\end{array}
			\right) \phi(z,\lambda). \nonumber
		\end{equation}
		Then \(\phi \) admits an Iwasawa factorization
		\begin{equation}
			\phi = F w B,\ \ \
			\ \ F \in (\Lambda {\rm SU}_{1,1} )_{\sigma }, \ \ B \in (\Lambda^{+, >0 } {\rm SL}_2 \mathbb{C} )_{\sigma },  \nonumber 
		\end{equation}
		where
		\begin{equation}
			w = Id\ \ \ {\rm or }\ \ \ \left( \begin{array}{cc}
				0 & \lambda \\
				-\lambda^{-1 } & 0
			\end{array} \right), \nonumber
		\end{equation}
		at \(z \) if and only if \(g \) admits a Birkhoff factorization
		\begin{equation}
			g = B_- \left( \begin{array}{cc}
				1 & 0 \\
				0 & -1
			\end{array} \right) B_+,\ \ \
			\ \ B_- \in (\Lambda^- {\rm SL}_2 \mathbb{C} )_{\sigma },\ \ B_+ \in (\Lambda^+ {\rm SL}_2 \mathbb{C} )_{\sigma }, \nonumber
		\end{equation}
		at \(z \).
	\end{prop}
	\begin{proof}
		Suppose that \(\phi \) admits an Iwasawa factorization \(\phi = FwB \) at \(z \). Let \( \epsilon = 1 \) if \(w = Id \) , \(\epsilon = -1 \) if \(w \neq Id \). We put \(B_+(\lambda) = B(z,\lambda) \) and \(B_-(\lambda) = \epsilon \overline{B(z, \overline{\lambda^{-1} } ) }^t \). Since \(B(z,\lambda) \in (\Lambda^{+, > 0 } {\rm SL}_2 \mathbb{C} )_{\sigma } \), we have \(B_+\in (\Lambda^+ {\rm SL}_2 \mathbb{C} )_{\sigma }, \ B_- \in (\Lambda^- {\rm SL}_2 \mathbb{C} )_{\sigma } \). From \(F \in (\Lambda {\rm SU}_{1,1})_{\sigma} \), we obtain a Birkhoff factorization
		\begin{align}
			g &= \overline{B(z,\overline{\lambda^{-1}} )}^t \overline{w}^t \left(\begin{array}{cc}
				1 & 0 \\
				0 & -1
			\end{array} \right) w B(z,\lambda) \nonumber \\
			&= \overline{B(z,\overline{\lambda^{-1}} )}^t \epsilon \left(\begin{array}{cc}
				1 & 0 \\
				0 & -1
			\end{array} \right) B(z,\lambda) \nonumber \\
			&= B_-(\lambda ) \left(
			\begin{array}{cc}
				1 & 0 \\
				0 & -1
			\end{array}
			\right) B(\lambda ). \nonumber
		\end{align}
		Conversely, if \(g \) admits a Birkhoff factorization 
		\begin{equation}
			g = B_- \left( \begin{array}{cc}
				1 & 0 \\
				0 & -1
			\end{array} \right) B_+,\ \ \ {\rm where }\ \ B_- \in (\Lambda^- {\rm SL}_2 \mathbb{C} )_{\sigma }, \ \ B_+ \in (\Lambda^+ {\rm SL}_2 \mathbb{C} )_{\sigma }, \nonumber
		\end{equation}
		at \(z \). Without loss of generality we can assume \(B_+ \in (\Lambda^{+, > 0 } {\rm SL}_2 \mathbb{C} )_{\sigma } \). \\
		From \(\overline{g(z,\overline{\lambda}^{-1})}^t = g(z,\lambda) \) and the uniqueness of the Birkhoff factorization, there exists a diagonal matrix \(\Theta = {\rm diag}(k,k^{-1}) \in {\rm SL}_2 \mathbb{R} \) such that
		\begin{equation}
			\overline{B_+(\overline{\lambda }^{-1 } ) }^t =B_-(\lambda ) \Theta, \ \ 
			B_+(\lambda ) = \Theta \overline{B_-(\overline{\lambda }^{-1 } ) }^t. \nonumber
		\end{equation}
		We put
		\begin{equation}
			B =  \left( \begin{array}{cc}
				\sqrt{|k| }^{-1 } & 0 \\
				0 & \sqrt{|k| }
			\end{array} \right) B_+,\ \ \ F = \phi B^{-1} w^{-1 } \nonumber
		\end{equation}
		where
		\begin{equation}
			w = \left\{ \begin{array}{cc}
				Id, & {\rm if }\ k > 0, \\
				\left( \begin{array}{cc}
					0 & \lambda \\
					- \lambda^{-1 } & 0
				\end{array} \right), & {\rm if }\ k < 0,
			\end{array}
			\right. \nonumber
		\end{equation}
		then we have \(B \in (\Lambda^{+, > 0 } {\rm SL}_2 \mathbb{C} )_{\sigma } \) and \(F \in (\Lambda {\rm SU}_{1,1} )_{\sigma } \). Hence, we obtain an Iwasawa factorization \(\phi = FwB \) at \(z \). \qed
	\end{proof}
	From Proposition 5.1, we investigate the Birkhoff factorization of \(g(z,\lambda) \) instead of the Iwasawa factorization of \(\phi \). First, we use the homogeneity condition on \(\phi \).
	\begin{lem}
		For \(z \in \mathbb{C} \backslash  (-\infty,0] \), \(g(|z|, \lambda) \) admits a Birkhoff factorization
		\begin{equation}
			g(|z|,\lambda) = B_- \left( \begin{array}{cc}
				1 & 0 \\
				0 & -1
			\end{array} \right) B_+,\ \ \ B_- \in (\Lambda^- {\rm SL}_2 \mathbb{C} )_{\sigma },\ \ B_+ \in (\Lambda^+ {\rm SL}_2 \mathbb{C} )_{\sigma }, \nonumber
		\end{equation}
		if and only if \(g(z,\lambda) \) admits a Birkhoff factorization
		\begin{equation}
			g(z,\lambda) = \tilde{B}_- \left( \begin{array}{cc}
				1 & 0 \\
				0 & -1
			\end{array} \right) \tilde{B}_+,\ \ \ \tilde{B}_- \in (\Lambda^- {\rm SL}_2 \mathbb{C} )_{\sigma },\ \ \tilde{B}_+ \in (\Lambda^+ {\rm SL}_2 \mathbb{C} )_{\sigma }. \nonumber
		\end{equation}
	\end{lem}
	\begin{proof}
		Suppose that \(g(|z|,\lambda) \) admits a Birkhoff factorization. We put \(r = |z| \in \mathbb{R}^* \) and \(e^{i\theta } = \frac{z}{|z|} \in S^1 \).
		From Proposition 3.3, we have \(\phi(z, \lambda ) = \delta T^{-1 } \phi(r,\lambda e^{-2i\theta }) T \). Since \(\delta \in (\Lambda {\rm SU}_{1,1 } )_{\sigma } \) and \(T \in U(1) \), we obtain \(g(z,\lambda ) =  T^{-1 } g(r,\lambda e^{-2i\theta } ) T \). We put \(\tilde{B}_-(z,\lambda) = T^{-1 } B_-(r,\lambda e^{-2i\theta } ) T \) and \(\tilde{B}_+(z,\lambda) = T^{-1 } B_+(r,\lambda e^{-2i\theta } ) T \), then we obtain a Birkhoff factorization of \(g(z,\lambda) \). \\
		The inverse is trivial. \qed
	\end{proof}
	From Lemma 5.1, without loss of generality we can assume \(z = r \in \mathbb{R}_{> 0} \). From now on, for fixed \(r \) we omit \(z \)-component in functions. For example, we write \(\phi(r,\lambda) = \phi(\lambda) \). 
	
	\subsection{The behaviour of \(g\)}
	Next, we investigate the behaviour of \(g \) by using Proposition 4.3. Let \(Pr : \mathbb{C} \rightarrow \mathbb{C}^* : \tilde{\lambda} \mapsto e^{\tilde{\lambda} }  = \lambda  \) be the universal covering of \(\mathbb{C}^* \) and we put \(\Omega_m = Pr(\tilde{\Omega }_m ) \subset \mathbb{C}^* \) for all \(m \in \mathbb{Z } \). Then
	\begin{equation}
		\Omega_0 = \left\{ \lambda  \in \mathbb{C}^* \ | \ - \frac{\pi}{2} < {\rm arg}(\lambda ) < \frac{3}{2}\pi \right\}, \nonumber
	\end{equation}
	and \(\Omega_m = e^{-im\pi } \Omega_0 \). We have \(\Omega_{m+2 } = \Omega_m \) fo all \(m \in \mathbb{Z} \).

	We use the result of Proposition 4.3 and we split \(g \) on \(\Omega_m \) so that the splitting looks like the Birkhoff factorization.
	\begin{prop}
		Fix \(r \in \mathbb{R}_{>0 } \), we define \(k_m:\Omega_m \rightarrow {\rm SL}_2 \mathbb{C}\ (m=0,1)\) by
		\begin{equation}
			k_0 = \tilde{k}_{2m} \circ (Pr|_{\tilde{\Omega }_{2m } } )^{-1} ,\ \ \ k_1 = \tilde{k}_{2m+1}  \circ (Pr|_{\tilde{\Omega }_{2m+1 } } )^{-1}. \nonumber
		\end{equation}
		where
		\begin{equation}
			\tilde{k}_m = F_0^{-1 } A_m^{-1 } H^{-1} \phi,\ \ \ \ \ m \in \mathbb{Z}, \nonumber 
		\end{equation}
		and
		\begin{equation}
			F_0 = \exp{\left( (\lambda^{-1 } + \lambda ) \frac{r^2}{2} \left( \begin{array}{cc}
					0 & 1 \\
					1 & 0
				\end{array} \right) \right) }. \nonumber
		\end{equation}
		Then we have
		\begin{align}
			g &= \overline{k_m(\overline{\lambda }^{-1 } ) }^t \left\{ \left(
			\begin{array}{cc}
				-1 & 0 \\
				0 & 1
			\end{array}
			\right) + 2\left(\frac{a^2}{\gamma^2} -1 \right) \frac{\gamma}{\pi} \left\{ (F_0)^2 + \left(
			\begin{array}{cc}
				0 & i \\
				-i & 0
			\end{array}
			\right) \right\} \right\} k_m(\lambda ), \nonumber
		\end{align}
		and \(k_m \) satisfies
		\begin{equation}
			k_m(\lambda) = (Id + o(\lambda) ) \left( \begin{array}{cc}
				r^{-1 } & 0\\
				0 & r
			\end{array} \right) \ \ {\rm as}\ \ \lambda \rightarrow 0 \ \ {\rm on}\ \Omega_m, \ \ \ m=0,1. \nonumber
		\end{equation}
	\end{prop}
	\begin{proof}
		Since \(\tilde{k}_{m+2}(\lambda e^{i 2\pi } ) = \tilde{k}_m(\lambda ) \) for \(\lambda = e^{\tilde{\lambda } }\ (\tilde{\lambda } \in \tilde{\Omega }_m) \), \(k_0,k_1\) are well-defined. From Proposition 4.3, for fixed \(r \) we have
		\begin{align}
			g &= C \overline{K_m(\overline{\lambda}^{-1} ) }^t \overline{\phi_0(\overline{\lambda}^{-1} ) }^t \left\{ \left(
			\begin{array}{cc}
				-1 & 0 \\
				0 & 1
			\end{array}
			\right) + 2\left(\frac{a^2}{\gamma^2} -1 \right) \frac{\gamma}{\pi} \left(
			\begin{array}{cc}
				1 & i \\
				-i & 1
			\end{array}
			\right) \right\} \nonumber \\
			&\hspace{6cm} \cdot \phi_0(\lambda) K_m(\lambda) C \nonumber \\
			&= C \overline{K_m(\overline{\lambda}^{-1} ) }^t \overline{B_0(\overline{\lambda}^{-1} ) }^{t -1} \nonumber \\
			&\ \ \ \ \ \cdot  \left\{ \left(
			\begin{array}{cc}
				-1 & 0 \\
				0 & 1
			\end{array}
			\right) + 2\left(\frac{a^2}{\gamma^2} -1 \right) \frac{\gamma}{\pi} \left\{ (F_0)^2 + \left(
			\begin{array}{cc}
				0 & i \\
				-i & 0
			\end{array}
			\right) \right\} \right\} \cdot \nonumber \\
			&\hspace{5cm} \cdot B_0(\lambda)^{-1} K_m(\lambda) C, \nonumber
		\end{align}
		where
		\begin{equation}
			B_0(r,\lambda) = \exp{\left( \lambda \frac{r^2}{2} \left( \begin{array}{cc}
					0 & 1 \\
					1 & 0
				\end{array} \right) \right) }. \nonumber
		\end{equation}
		Since
		\begin{align}
			\tilde{k}_m(\lambda) &= B_0(\lambda)^{-1} K_m(\lambda) C \nonumber \\ 
			&= F_0^{-1} A_m^{-1} H^{-1} \phi \nonumber \\
			&= (Id + o(\lambda) ) C \ \ {\rm as}\ \ \lambda \rightarrow 0 \ \ {\rm on}\ \tilde{\Omega }_m, \nonumber
		\end{align}
		we obtain the stated results. \qed
	\end{proof}

	From Proposition 5.2 and observation of Remark 4.3, if we choose \(a = \gamma \)
	\begin{equation}
		g(\lambda) = \overline{k_m(\overline{\lambda}^{-1} ) }^t \left(
		\begin{array}{cc}
			-1 & 0 \\
			0 & 1
		\end{array}
		\right) k_m(\lambda) \ \ \ {\rm on} \ \ \Omega_m,\ \ \ (m = 0,1), \nonumber
	\end{equation}
	where \(k_m = (Id + o(\lambda) )C \) as \(\lambda \rightarrow 0 \) on \(\Omega_m \).\vspace{2mm}
	
	This appears to have the form of the Birkhoff factorization of \(g \), but in fact \(k_m \) is only defined on \(\Omega_m \) (and is multivalued on \(\mathbb{C}^* \)). However, the above calculation shows that \(a = \gamma \) is a ``good candidate'' for constructing a Birkhoff factorization. From this observation, we construct \(B_- \in (\Lambda^- {\rm SL}_2 \mathbb{C})_{\sigma} \) and \(B_+ \in (\Lambda^+ {\rm SL}_2 \mathbb{C})_{\sigma} \) of the Birkhoff factorization \(g = B_- {\rm diag}(1,-1)B_+ \) by using \(k_0,k_1 \). \vspace{1mm}
	
	To summarize so far,
	\begin{enumerate}
		\item [(R1)] we showed the equivalence between the Iwasawa factorization of \(\phi\) and the Birkhoff factorization of \(g\),\vspace{2mm}
		
		\item [(R2)] we found a ``good candidate'' \(a=\gamma\) for giving the Birkhoff factorization and gave a ``local factoirzation'' \(g(\lambda) = \overline{k_m(\overline{\lambda}^{-1})}^t {\rm diag}(-1,1) k_m(\lambda)\),
	\end{enumerate}\vspace{2mm}
	
	From the results (R1), (R1), we give the Iwasawa factorization of \(\phi\). Our strategy to give the Iwasawa factorization is given as follows.
	\begin{enumerate}
		\item [(S1)] under the assumption of the existence of the Birkhoff factorization, we investigate the properties of \(h_{\pm} = B_+k_m^{-1}\),\vspace{1mm}
		
		\item [(S2)] we show the inverse of (a), i.e. if there exists sectionally-holomorphic functions \(h_{\pm}\) such that \(h_{\pm}\) satisfy the properties in (a) then \(B_+ = h_{\pm}k_m\) give the Birkhoff factorization of \(g \),\vspace{1mm}
		
		\item [(S3)] we show that the existence of \(h_{\pm}\) in (b) is equivalent to the existence of solution of certain Riemann-Hilbert problem,\vspace{1mm}
		
		\item [(S4)] we solve the Riemann-Hilbert problem by using the Vanishing Lemma,\vspace{1mm}
		
		\item [(S5)] we construct the desired \(h_{\pm}\) from the solution of Riemann-Hilbert problem.
	\end{enumerate}\vspace{2mm}
	
	First, we consider the problem (S1). The difference between \(k_0,k_1\) and \(B_+\) is given as follows.
	\begin{prop}
		If \(g \) admits a Birkhoff factorization
		\begin{equation}
			g = B_- \left(\begin{array}{cc}
				1 & 0 \\
				0 & -1
			\end{array}\right) B_+,\ \ \ \ \ B_+ \in (\Lambda^{+,>0} {\rm SL}_2 \mathbb{C})_{\sigma}, \nonumber
		\end{equation}
		for \(\gamma = a\), then
		\begin{align}
			&h_+ = B_+k_0^{-1}:\Omega_0 \rightarrow {\rm SL}_n\mathbb{C},\ \ \ h_- = B_+k_1^{-1}:\Omega_1 \rightarrow {\rm SL}_n\mathbb{C}, \nonumber
		\end{align} 
		satisfiy the following conditions:
		\begin{enumerate}
			\item [(i)] \(h_+ k_0 = h_- k_1 \) on \(\Omega_0 \cap \Omega_1 \),\vspace{1mm}
			
			\item [(ii)] \(h_{\pm} k_m \) satisfies the twisted condition \((h_{\pm } k_m)(-\lambda) = \sigma((h_{\mp } k_m)(\lambda)) \),\vspace{1mm}
			
			\item [(iii)] \(h_{\pm } k_m \) can be extended to \(\lambda = 0 \),\vspace{1mm}
			
			\item [(iv)] \(g(\lambda) = \overline{(h_{\pm} k_m)(\overline{\lambda }^{-1} ) }^t \left( \begin{array}{cc}
				k^{-1} & 0 \\
				0 & -k
			\end{array} \right) (h_{\pm} k_m)(\lambda),\ \ \ k \in \mathbb{R}\).
		\end{enumerate}
	\end{prop}
	\begin{proof}
		Since \(B_+ \in (\Lambda^+ {\rm SL}_2 \mathbb{C})_{\sigma}\), we obtain (i), (ii), (iii). From the proof of Proposition 5.1, we have
		\begin{equation}
			\overline{B_+(\overline{\lambda}^{-1})}^t = B_-(\lambda)\left(\begin{array}{cc}
				k & 0 \\
				0 & k^{-1}
			\end{array}\right),\ \ \ \ \ k \in \mathbb{R}^*. \nonumber
		\end{equation}
		Thus, we obtain (iv).
		\qed
	\end{proof}
	\begin{rem}
		If \(g\) admits a Birkhoff factorization \(g=B_-{\rm diag}(1,-1)B_+\), \(B_+ \in (\Lambda^+ {\rm SL}_2 \mathbb{C})_{\sigma}\) then \(g\) admits a Birkhoff factorization \(g=\tilde{B}_-{\rm diag}(1,-1)\tilde{B}_+\), \(\tilde{B}_+ \in (\Lambda^{+,>0} {\rm SL}_2 \mathbb{C})_{\sigma}\).
	\end{rem}
	\begin{rem}
		In Proposition 5.3, we put
		\begin{equation}
			\tilde{h}_{\pm} = \left(\begin{array}{cc}
				\sqrt{|k|^{-1}} & 0 \\
				0 & \sqrt{|k|}
			\end{array}\right)h_{\pm}, \nonumber
		\end{equation}
		then \(\tilde{h}_{\pm}\) satisfy the conditions (i), (ii), (iii) and
		\begin{equation}
			g = \frac{k}{|k|}\overline{(\tilde{h}_{\pm} k_m)(\overline{\lambda }^{-1} ) }^t \left( \begin{array}{cc}
				1 & 0 \\
				0 & -1
			\end{array} \right) (\tilde{h}_{\pm} k_m)(\lambda). \nonumber
		\end{equation}
		Thus, without loss of generality we can assume \(K \in \{-1,1\}\) in (iv) of Proposition 5.3.
	\end{rem}
	
	\subsection{The Riemann-Hilbert problem from the Birkhoff factorization}
	We consider the problem (S2). Given holomorphic functions \(h_{\pm}\) which satisfies the conditions in Proposition 5.3 then we obtain the global Iwasawa factorization as follows. Here, the conditions (1), (2), (3), (4) in Proposition 5.4 are induced by the conditions in Proposition 5.3 and the propserties of \(k_0, k_1\).
	
	
	
	
	\begin{prop}
		Let \(H_+, H_- \) be the upper-half plane and lower-half plane of \(\mathbb{C}^* \) respectively. We put
		\begin{equation}
			\Gamma_+ = \{\lambda \in \mathbb{C}^*\ | \ {\rm Re}(\lambda) > 0 \},\ \ \Gamma_- = \{\lambda \in \mathbb{C}^*\ | \ {\rm Re}(\lambda) < 0 \}. \nonumber
		\end{equation}
		
		Fix \(z = r \in \mathbb{R}_{> 0 } \) and we choose \(a = \gamma \). Suppose that there exist holomorphic maps \(h_{\pm } : H_{\pm } \rightarrow {\rm SL}_2 \mathbb{C} \) such that \(h_{\pm } \) are continuous in \(\overline{H}_{\pm } \) and \(h_{\pm } \) satisfy the conditions
		\begin{enumerate}
			\item[(1)] \(h_+ k_0 = h_- k_1 \) on \(\Gamma_+ \cup \Gamma_- \),
			\item[(2)]  \( h_{\pm }(\lambda e^{i\pi } ) = \left(
			\begin{array}{cc}
				1 & 0 \\
				0 & -1
			\end{array}
			\right) h_{\mp }(\lambda) \left(
			\begin{array}{cc}
				1 & 0 \\
				0 & -1
			\end{array}
			\right) \),
			\item[(3)] \(h_{\pm } \rightarrow \left(
			\begin{array}{cc}
				\rho & 0 \\
				0 & \rho^{-1}
			\end{array}
			\right) \) as \(\lambda \rightarrow 0 \) on \(\overline{H}_{\pm } \), where \(\rho \in \mathbb{R}^* \),
			\item [(4)] \( \overline{ h_{\pm }(\overline{\lambda}^{-1 } ) }^t \left(
			\begin{array}{cc}
				1 & 0 \\
				0 & -1
			\end{array}
			\right) h_{\pm }(\lambda) = \epsilon \left(
			\begin{array}{cc}
				1 & 0 \\
				0 & -1
			\end{array}
			\right),\ \ \ \epsilon \in \{-1, 1 \} \),
		\end{enumerate}
		then \(g \) admits a Birkhoff factorization
		\begin{equation}
			g = B_- \left( \begin{array}{cc}
				1 & 0 \\
				0 & -1
			\end{array} \right) B_+,\ \ \ {\rm where}\ \ B_- \in (\Lambda^- {\rm SL}_2 \mathbb{C} )_{\sigma },\ \ B_+ \in (\Lambda^+ {\rm SL}_2 \mathbb{C} )_{\sigma }, \nonumber
		\end{equation}
		at \(z = r \).
	\end{prop}
	\begin{proof}
		From (1), we define a continuous map \(B_+ : \mathbb{C}^* \rightarrow {\rm SL}_2 \mathbb{C} \)  by
		\begin{align}
			B_+ = \left\{ \begin{array}{cc}
				h_+ k_0 & {\rm on}\ \overline{H}_+, \\
				h_- k_1 & {\rm on}\ \overline{H}_-.
			\end{array}  \right. \nonumber
		\end{align}
		Since \(h_+ k_0, h_- k_1 \) are holomorphic on \(H_+, H_- \) and continuous on \(\overline{H}_+, \overline{H }_- \) respectively, from Painleve's Theorem we obtain \(B_+ \in \Lambda {\rm SL }_2 \mathbb{C} \).
		For \(\lambda \in \overline{H }_+ \), from (2) we obtain
		\begin{align}
			B_+(\lambda e^{i\pi } ) &= h_+(\lambda e^{i\pi } ) k_0(\lambda e^{i\pi } ) \nonumber \\
			&= \left(
			\begin{array}{cc}
				1 & 0 \\
				0 & -1
			\end{array}
			\right) h_-(\lambda) \left(
			\begin{array}{cc}
				1 & 0 \\
				0 & -1
			\end{array}
			\right) \left(
			\begin{array}{cc}
				1 & 0 \\
				0 & -1
			\end{array}
			\right) k_1(\lambda) \left(
			\begin{array}{cc}
				1 & 0 \\
				0 & -1
			\end{array}
			\right) \nonumber \\
			&= \left(
			\begin{array}{cc}
				1 & 0 \\
				0 & -1
			\end{array}
			\right) B_+(\lambda) \left(
			\begin{array}{cc}
				1 & 0 \\
				0 & -1
			\end{array}
			\right). \nonumber
		\end{align}
		Similarly, for \(\lambda \in \overline{H }_- \) we obtain
		\begin{equation}
			B_+(\lambda e^{i\pi } ) = \left(
			\begin{array}{cc}
				1 & 0 \\
				0 & -1
			\end{array}
			\right) B_+(\lambda) \left(
			\begin{array}{cc}
				1 & 0 \\
				0 & -1
			\end{array}
			\right). \nonumber
		\end{equation}
		Thus, we have \(B_+ \in (\Lambda {\rm SL}_2 \mathbb{C} )_{\sigma } \).
		
		From the condition (3),
		\begin{align}
			B_+ &= h_{\pm } k_m \nonumber \\
			&= (Id + o(\lambda))  \left( \begin{array}{cc}
				\rho r^{-1 } & 0 \\
				0 & \rho^{-1 } r
			\end{array} \right)\ \ \ {\rm as}\ \ \lambda \rightarrow 0\ \ {\rm on}\ \overline{H}_{\pm },\ \ \ (m = 0,1), \nonumber
		\end{align}
		and then \(B_+ \in (\Lambda^+ {\rm SL}_2 \mathbb{C} )_{\sigma } \). We put \(B_-(\lambda ) = \epsilon \overline{B_+(\overline{\lambda }^{-1} ) }^t \in (\Lambda^- {\rm SL }_2 \mathbb{C} )_{\sigma } \), then from (4) and Proposition 5.2, we have
		\begin{align}
			g(\lambda) &= \overline{k_m(\overline{\lambda}^{-1} ) }^t \left(
			\begin{array}{cc}
				1 & 0 \\
				0 & -1
			\end{array}
			\right) k_m(\lambda) \nonumber \\
			&= \epsilon \overline{k_m(\overline{\lambda}^{-1} ) }^t  \overline{ h_{\pm }(1/\overline{\lambda} ) }^t \left(
			\begin{array}{cc}
				1 & 0 \\
				0 & -1
			\end{array}
			\right) h_{\pm }(\lambda) k_m(\lambda) \nonumber \\
			&= \epsilon \overline{B_+(\overline{\lambda }^{-1} ) }^t \left(
			\begin{array}{cc}
				1 & 0 \\
				0 & -1
			\end{array}
			\right) B_+(\lambda ) \nonumber \\
			&= B_-(\lambda) \left( \begin{array}{cc}
				1 & 0 \\
				0 & -1
			\end{array} \right) B_+(\lambda ). \nonumber
		\end{align}
		Hence, we obtain a Birkhoff factorization of \(g \) at \(z = r \). \qed
	\end{proof}
	From Proposition 
	5.4, we obtain a global Birkhoff factorization if we find \(h_{\pm } \) which satisfy (1), (2), (3), (4). Thus, it is enough to seek \(h_{\pm} \) in Proposition 5.3 for giving the Birkhoff factorization of \(g \). To simplify the problem, we consider the following setting. It is the problem (S3).\vspace{2mm}
	
	Suppose that there exists a solution \(h = \{h_{\pm } \} \) in Proposition 5.3, we put
	\begin{equation}
		Y_+ = \epsilon \left(
		\begin{array}{cc}
			\rho^{-1} & 0 \\
			0 & \rho
		\end{array}
		\right) (h_+^t)^{-1},\ \ \ Y_- = \epsilon \left(
		\begin{array}{cc}
			\rho^{-1} & 0 \\
			0 & \rho
		\end{array}
		\right) (h_-^t)^{-1},\ \ \ \epsilon \in \{-1,1 \}, \nonumber
	\end{equation}
	and
	\begin{align}
		G &= (k_0 k_1^{-1} )^t \nonumber \\
		&= \left\{ \begin{array}{cc}
			\left(
			\begin{array}{cc}
				1 + ie^{- \frac{r^2}{2} (\lambda^{-1} + \lambda ) } & -i e^{- \frac{r^2}{2} (\lambda^{-1} + \lambda ) } \\
				i e^{- \frac{r^2}{2} (\lambda^{-1} + \lambda ) } & 1 - ie^{- \frac{r^2}{2} (\lambda^{-1} + \lambda ) }
			\end{array}
			\right) & {\rm on}\ \ \Gamma_+, \\
			\left(
			\begin{array}{cc}
				1 - ie^{\frac{r^2}{2} (\lambda^{-1} + \lambda ) } & -i e^{\frac{r^2}{2} (\lambda^{-1} + \lambda ) } \\
				i e^{\frac{r^2}{2} (\lambda^{-1} + \lambda ) } & 1 + ie^{\frac{r^2}{2} (\lambda^{-1} + \lambda ) }
			\end{array}
			\right)& {\rm on}\ \ \Gamma_-.  \end{array} \right. \nonumber
	\end{align}
	Since \(h_{\pm } \) satisfy the condition (1), the condition (2) is equivalent to
	\begin{enumerate}
		\item[(2)']	\( h_{\pm } \rightarrow \epsilon \left(
		\begin{array}{cc}
			\rho^{-1} & 0 \\
			0 & \rho
		\end{array}
		\right)\ {\rm as}\ \lambda \rightarrow \infty, \ \ \ \epsilon \in \{-1,1 \} \). 
	\end{enumerate}
	From the conditions (1), (2)', (3), \(Y = \{Y_{\pm} \} \) satisfy
	\begin{enumerate}
		\item[(i)] \(Y_{\pm } \) are holomorphic in \(H_{\pm } \) and continuous in \(\overline{H_{\pm } } \) respectively.
		\vspace{2mm} 
		
		\item[(ii)] \(Y_+ = Y_- G \) on \(\Gamma = \Gamma_+ \cup \Gamma_- \), \vspace{2mm}
		
		\item[(iii)]  \(Y_{\pm } \rightarrow Id \) as \(\lambda \rightarrow \infty \).
	\end{enumerate}
	This is the Riemann-Hilbert problem for \(G \) on \(\Gamma \) in the manner of chapter 3 of \cite{FIKN2006}. If there exists a solution of the Riemann-Hilbert problem, then we construct \(h_{\pm} \) satisfying (1), (2), (3) in Proposition 
	5.4 from the solution \(Y = \{Y_{\pm}\} \). Furthermore, we prove that \(h_{\pm} \) also satisfy (4) in Proposition 
	5.4. \vspace{3mm}
	
	\subsection{The Iwasawa factorization from the Rimann-Hilbert problem}
	We consider the problem (S4). To obtain a solution of this problem, we use the Vanishing Lemma (Corollary 3.2 of \cite{FIKN2006}):
	\begin{lem} [The Vanishing Lemma]
		The Riemann Hilbert problem determined by a pair \((\Gamma, G) \) has a solution if and only if the corresponding homogenous Riemann-Hilbert problem (in which the condition \(Y \rightarrow Id \) is replaced by \(Y \rightarrow 0 \)) has only the trivial solution. 
	\end{lem}
	To use the Vanishing Lemma, we show the following Lemma by the same way as section 5 of \cite{GIL20152}:
	\begin{lem}
		Let \(Y_{\pm } \) be a solution of the homogenous Riemann-Hilbert problem for \(G \) on \(\Gamma \), then
		\begin{enumerate}
			\item [(a)] \(\int_{\Gamma} Y_+(\lambda) \overline{Y_-(\overline{\lambda} ) }^t d\lambda = 0  \),
			\item [(b)] \(\int_{\Gamma} Y_-(\lambda) \overline{Y_+(\overline{\lambda} ) }^t d\lambda = 0  \).
		\end{enumerate}  
	\end{lem}
	\begin{proof}
		Since \(G \rightarrow Id \) exponentially as \(\lambda \rightarrow \infty \), \( Y_+ \) and \(Y_- \) also do so. \(Y_+(\lambda) \), \( \overline{Y_-(\overline{\lambda}^{-1} )}^t \) are holomorphic on \( {\rm Im}(\lambda > 0) \), thus by using Cauchy's Theorem, we obtain (a). Similarly, we prove (b). \qed
	\end{proof}
	From the Vanishing Lemma and Lemma 5.2, we prove:
	\begin{cor}
		There is a solution \(Y = \{Y_{\pm } \} \) of the Riemann-Hilbert problem for \(G \) on \(\Gamma \). 
	\end{cor}
	\begin{proof}
		Let \(Y_{\pm } \) be a solution of the homogenous Riemann-Hilbert problem for \(G \) on \(\Gamma \),then we have
		\begin{align}
			N &= G(\lambda) + \overline{G(\overline{\lambda} )}^t \nonumber \\
			&= \left\{ \begin{array}{cc} \left(
				\begin{array}{cc}
					2 & -2i e^{-\frac{r^2}{2} (\lambda^{-1} + \lambda ) } \\
					2ie^{-\frac{r^2}{2} (\lambda^{-1} + \lambda )  } & 2
				\end{array}
				\right) & {\rm on } \ \ \Gamma_+, \\
				\left(
				\begin{array}{cc}
					2 & -2i e^{\frac{r^2}{2} (\lambda^{-1} + \lambda ) } \\
					2ie^{\frac{r^2}{2} (\lambda^{-1} + \lambda )  } & 2
				\end{array}
				\right) & {\rm on } \ \ \Gamma_-.   \end{array} \right. \nonumber
		\end{align} Since \(N \) are positive definite, from Lemma 5.2 we obtain
		\(Y = 0 \). By using the Vanishing Lemma, we obtain the solution of the Riemann Hilbert problem for \(G \) on \(\Gamma \). \qed
	\end{proof}
	We follow section 3 of \cite{FIKN2006} and prove the following Lemmas:
	\begin{lem}
		Let \(Y \) be a solution of Riemann-Hilbert problem for (i), (ii), (iii). \\
		Then \(Y \) is unique.
	\end{lem}
	\begin{proof}
		Let \(\tilde{Y} \) be an another solution of Riemann-Hilbert problem with \\
		\(\lim_{\lambda \rightarrow \infty} \tilde{Y } = Id \) for \(G \) on \(\Gamma \). We put \(X = \tilde{Y}^{-1} Y \), then on \(\Gamma \) 
		\begin{equation}
			X_+ = \tilde{Y}_+^{-1} Y_+ = \tilde{Y}_-^{-1 } Y_- = X_-, \nonumber
		\end{equation}
		and \(X \) is holomorphic on \(\mathbb{C}^* \). Since \(X \rightarrow Id \ \ {\rm as}\ \lambda \rightarrow \infty \), by using the Liouville's theorem we obtain \(X = Id \) i.e. \(\tilde{Y} = Y \). \qed
	\end{proof}
	From the uniqueness of solution, we prove:
	\begin{lem}
		We have
		\begin{equation}
			Y(0) = \left(
			\begin{array}{cc}
				a & 0 \\
				0 & a^{-1}
			\end{array}
			\right),\ \ \ {\rm where }\ \ a \in \mathbb{R}^*. \nonumber
		\end{equation}
	\end{lem}
	\begin{proof}
		From \(\lim_{\lambda \rightarrow 0} G = Id \), there exists \(Y(0) \).
		Since \(\overline{G(\overline{\lambda}^{-1} ) } = G(\lambda)^{-1} \), we have
		\begin{equation}
			\overline{Y(\overline{\lambda}) } = \overline{Y(\overline{\lambda}) G(\overline{\lambda} ) } = \overline{Y(-\overline{\lambda})} G(\lambda )^{-1}, \nonumber 
		\end{equation}
		and then from the uniqueness of \(Y \), we obtain \( \overline{Y_{\pm}(\overline{\lambda}) } = Y_{\mp}(\lambda) \). Thus,  we have \( \overline{Y(0)} = Y(0) \in {\rm SL}_2 \mathbb{R} \). Similarly, from
		\begin{equation}
			\left(
			\begin{array}{cc}
				1 & 0 \\
				0 & -1
			\end{array}
			\right) G(-\lambda ) \left(
			\begin{array}{cc}
				1 & 0 \\
				0 & -1
			\end{array}
			\right) = G(\lambda)^{-1}. \nonumber
		\end{equation}
		we have
		\begin{equation}
			\left( \begin{array}{cc}
				1 & 0 \\
				0 & -1
			\end{array} \right) Y(-\lambda) \left( \begin{array}{cc}
				1 & 0 \\
				0 & -1
			\end{array} \right) = Y(\lambda). \nonumber
		\end{equation}
		Hence, we obtain the stated result. \qed
	\end{proof}
	\begin{lem}
		We have
		\begin{enumerate}
			\item[(I)] \(Y(\lambda e^{i\pi } ) = \left( \begin{array}{cc}
				1 & 0 \\
				0 & -1
			\end{array} \right) Y(\lambda) \left( \begin{array}{cc}
				1 & 0 \\
				0 & -1
			\end{array} \right) \).
			\item[(II)] \(Y(\lambda) = \left(
			\begin{array}{cc}
				a & 0 \\
				0 & -a^{-1}
			\end{array}
			\right) \left(\overline{Y(\overline{\lambda}^{-1} )}^t \right)^{-1} \left(
			\begin{array}{cc}
				1 & 0 \\
				0 & -1
			\end{array}
			\right) \). \\
		\end{enumerate}
	\end{lem}
	\begin{proof}
		By Lemma 5.5, we obtain (I) from 
		\begin{equation}
			G(\lambda e^{i\pi } ) = \left( \begin{array}{cc}
				1 & 0 \\
				0 & -1
			\end{array} \right) G(\lambda)^{-1 } \left( \begin{array}{cc}
				1 & 0 \\
				0 & -1
			\end{array} \right). \nonumber
		\end{equation}
		and we obtain (II) from
		\begin{equation}
			\left(
			\begin{array}{cc}
				1 & 0 \\
				0 & -1
			\end{array}
			\right) \left( \overline{G(\overline{\lambda}^{-1} ) }^t \right)^{-1} \left(
			\begin{array}{cc}
				1 & 0 \\
				0 & -1
			\end{array}
			\right) = G(\lambda), \nonumber
		\end{equation}
		\qed
	\end{proof}
	
	Finally, we consider the problem (S5). From Proposition 5.1, 5.2, 5.4, Corollary 5.5 and Lemma 5.6, we obtain the global Iwasawa factorization of \(\phi \):
	\begin{thm}
		\begin{equation}
			\phi = \left(
			\begin{array}{cc}
				\sqrt{-\gamma} & -\lambda / \sqrt{-\gamma} \\
				0 & 1 / \sqrt{-\gamma}
			\end{array}
			\right) 
			\left(
			\begin{array}{cc}
				1 & 0 \\
				\lambda^{-1} \log{\frac{z}{2}} & 1
			\end{array}
			\right) \left(
			\begin{array}{cc}
				y_0 & \lambda z (y_0)_z \\
				\frac{1}{4} y_1 & y_0 + \lambda z \frac{1}{4} (y_1)_z
			\end{array}
			\right), \nonumber
		\end{equation}
		where \(\gamma \) is the Euler constant and
		\begin{equation}
			y_0(z) = \sum_{j \ge 0} \frac{\lambda^{-2j} z^{4j} }{16^j(j!)^2} ,\ \ y_1(z) = -2 \lambda^{-1} \sum_{j \ge 1} (1 + \cdots + \frac{1}{j} ) \frac{\lambda^{-2j} z^{4j} }{16^j (j!)^2}, \nonumber  
		\end{equation}
		admits an Iwasawa factorization
		\begin{equation*}
			\phi = F \left(
			\begin{array}{cc}
				0 & \lambda \\
				-\lambda^{-1} & 0
			\end{array}
			\right) B, \nonumber
		\end{equation*}
		on \(\mathbb{C} \backslash  (-\infty,0] \).
	\end{thm}
	\begin{proof}
		From Lemma 5.1, we can assume \(z = r \in \mathbb{R}_{> 0 } \) and we fix \(r \). From Corollary 5.5, there exists a solution \(Y = \{Y_{\pm } \} \) of the Riemann-Hilbert problem (i), (ii), (iii). We put
		\begin{align}
			&h_+(\lambda) = \left(
			\begin{array}{cc}
				\sqrt{|a| } & 0 \\
				0 & \sqrt{|a| }^{-1}
			\end{array}
			\right) \left(Y_+(\lambda)^t \right)^{-1}, \nonumber \\
			&h_-(\lambda)  = \left(
			\begin{array}{cc}
				\sqrt{|a| } & 0 \\
				0 & \sqrt{|a| }^{-1}
			\end{array}
			\right) \left(Y_-(\lambda)^t \right)^{-1}, \nonumber
		\end{align}
		where \(Y(0) = {\rm diag}(a, a^{-1 }) \). We prove that \(h_{\pm } \) satisfy the conditions
		\begin{enumerate}
			\item[(1)] \(h_+ k_0 = h_- k_1 \) on \(\Gamma_+ \cup \Gamma_- \),
			\item[(2)]  \( h_{\pm }(\lambda e^{i\pi } ) = \left(
			\begin{array}{cc}
				1 & 0 \\
				0 & -1
			\end{array}
			\right) h_{\mp }(\lambda) \left(
			\begin{array}{cc}
				1 & 0 \\
				0 & -1
			\end{array}
			\right) \).
			\item[(3)] \(h_{\pm } \rightarrow \left(
			\begin{array}{cc}
				\rho & 0 \\
				0 & \rho^{-1}
			\end{array}
			\right) \) as \(\lambda \rightarrow 0 \) on \(\overline{H}_{\pm } \), where \(\rho \in \mathbb{R}^* \),
			\item [(4)] \( \overline{ h_{\pm }(\overline{\lambda}^{-1 } ) }^t \left(
			\begin{array}{cc}
				1 & 0 \\
				0 & -1
			\end{array}
			\right) h_{\pm }(\lambda) = \epsilon \left(
			\begin{array}{cc}
				1 & 0 \\
				0 & -1
			\end{array}
			\right), \ \ \ \epsilon \in \{-1,1 \} \).
		\end{enumerate}
		From (ii) and \(G = (k_0 k_1^{-1 } )^t \), we have
		\begin{align}
			h_+ k_0 &= \left( \begin{array}{cc}
				\sqrt{|a| } & 0 \\
				0 & \sqrt{|a| }^{-1}
			\end{array}
			\right) \left( Y_+^t \right)^{-1} k_0 \nonumber \\
			&= \left( \begin{array}{cc}
				\sqrt{|a| } & 0 \\
				0 & \sqrt{|a| }^{-1}
			\end{array}
			\right) \left( Y_-^t \right)^{-1} \left( G^t \right)^{-1} k_0 \nonumber \\
			&=\left( \begin{array}{cc}
				\sqrt{|a| } & 0 \\
				0 & \sqrt{|a| }^{-1}
			\end{array}
			\right) \left( Y_-^t \right)^{-1} k_1 = h_- k_1, \ \ \ {\rm on}\ \ \Gamma_+ \cup \Gamma_-. \nonumber
		\end{align}
		Thus, we obtain (1).

		From (I) of Lemma 5.6, we have
		\begin{align}
			h_{\pm }(\lambda e^{i\pi }) &= \left(
			\begin{array}{cc}
				\sqrt{|a| } & 0 \\
				0 & \sqrt{|a| }^{-1 }
			\end{array}
			\right) \left( Y_{\pm }(\lambda e^{i\pi } )^t \right)^{-1} \nonumber \\
			&= \left(
			\begin{array}{cc}
				\sqrt{|a| } & 0 \\
				0 & \sqrt{|a| }^{-1 }
			\end{array}
			\right) \left( \begin{array}{cc}
				1 & 0 \\
				0 & -1
			\end{array} \right) \left( Y_{\mp }(\lambda)^t \right)^{-1} \left( \begin{array}{cc}
				1 & 0 \\
				0 & -1
			\end{array} \right) \nonumber \\
			&= \left( \begin{array}{cc}
				1 & 0 \\
				0 & -1
			\end{array} \right) h_{\mp }(\lambda) \left( \begin{array}{cc}
				1 & 0 \\
				0 & -1
			\end{array} \right), \nonumber
		\end{align}
		then we obtain (2).
		
		From (iii) and (II) of Lemma 5.6, we have
		\begin{align}
			h_{\pm }(\lambda ) &= \left( \begin{array}{cc}
				\sqrt{|a| } & 0 \\
				0 & \sqrt{|a| }^{-1}
			\end{array}
			\right) \left( Y_{\pm }(\lambda )^t \right)^{-1} \nonumber \\
			&= \left( \begin{array}{cc}
				\sqrt{|a| } & 0 \\
				0 & \sqrt{|a| }^{-1}
			\end{array}
			\right)
			\left( \begin{array}{cc}
				a^{-1 } & 0 \\
				0 & -a
			\end{array} \right) \overline{Y_{\pm }(\overline{\lambda }^{-1 } ) } \left( \begin{array}{cc}
				1 & 0 \\
				0 & -1
			\end{array} \right) \nonumber \\
			&\rightarrow \frac{a}{|a|} \left( \begin{array}{cc}
				\sqrt{|a| }^{-1 } & 0 \\
				0 & \sqrt{|a| } 
			\end{array} \right) \ \ \ {\rm as}\ \ \lambda \rightarrow 0, \ \ (m = 0,1). \nonumber
		\end{align}
		Let \(\rho = \frac{a}{|a|} \sqrt{|a| }^{-1} \in \mathbb{R}^* \), then we obtain (3).
		
		From (II) of Lemma 5.6, we have
		\begin{align}
			\overline{h_{\pm }(\overline{\lambda}^{-1} ) }^t 
			&= \overline{Y_{\pm }(\overline{\lambda}^{-1} ) }^{-1} \left(
			\begin{array}{cc}
				\sqrt{|a| } & 0 \\
				0 & \sqrt{|a| }^{-1 }
			\end{array}
			\right) \nonumber \\
			&= \left(
			\begin{array}{cc}
				1 & 0 \\
				0 & -1
			\end{array}
			\right) Y_{\pm }(\lambda)^t \left(
			\begin{array}{cc}
				a^{-1} & 0 \\
				0 & -a
			\end{array}
			\right) \left(
			\begin{array}{cc}
				\sqrt{|a| } & 0 \\
				0 & \sqrt{|a| }^{-1 }
			\end{array}
			\right)  \nonumber \\
			&= \frac{a}{|a|} \left(
			\begin{array}{cc}
				1 & 0 \\
				0 & -1
			\end{array}
			\right) h_{\pm }(\lambda)^{-1} \left(
			\begin{array}{cc}
				1 & 0 \\
				0 & -1
			\end{array}
			\right). \nonumber
		\end{align}
		Let \(\epsilon = \frac{a}{|a|} \), then we obtain (4).
		
		Thus, from Proposition 5.1 and Proposition 5.4, \(\phi \) admits an Iwasawa factorization at \(z = r\). Since we choose an arbitrary \(z \), we obtain the global Iwasawa factorization of \(\phi  \) on \(\mathbb{C} \backslash  (-\infty,0] \).
		From Proposition 3.2 and the Iwasawa factorization near \(z = 0 \), we have
		\begin{equation}
			\phi = F \left(
			\begin{array}{cc}
				0 & \lambda \\
				-\lambda^{-1} & 0
			\end{array}
			\right) B. \nonumber
		\end{equation}
		Hence, we obtain the stated result. \qed
	\end{proof}
	Let
	\begin{align}
		&\left(
		\begin{array}{cc}
			\sqrt{-\gamma} & -\lambda / \sqrt{-\gamma} \\
			0 & 1 / \sqrt{-\gamma}
		\end{array}
		\right) 
		\left(
		\begin{array}{cc}
			1 & 0 \\
			\lambda^{-1} \log{\frac{z}{2}} & 1
		\end{array}
		\right) \left(
		\begin{array}{cc}
			y_0 & \lambda z (y_0)_z \\
			\frac{1}{4} y_1 & y_0 + \lambda z \frac{1}{4} (y_1)_z
		\end{array}
		\right) \nonumber \\
		&\hspace{5cm} =  F \left(
		\begin{array}{cc}
			0 & \lambda \\
			-\lambda^{-1} & 0
		\end{array}
		\right) B, \nonumber
	\end{align}
	be the Iwasawa factorization and
	\begin{equation}
		B|_{\lambda  = 0 } = \left( \begin{array}{cc}
			e^{\frac{v}{2} } & 0 \\
			0 & e^{-\frac{v}{2} }
		\end{array} \right). \nonumber 
	\end{equation}
	We put \(e^u = |z|^2 e^v \) and \(t = \frac{z^2}{2} \), then from Corollary 3.4 and Theorem 5.6, we obtain the global solution \(u : \mathbb{C}^* \rightarrow \mathbb{R} \) of the sinh-Gordon equation
	\begin{equation}
		u_{t \overline{t} } = e^{2u } - e^{-2u }, \nonumber
	\end{equation}
	with \(u \sim (1 + o(t)) \log{|t|} \) as \(t \rightarrow 0 \).
	
	\appendix
	\section{Proofs of section 3}
	\def\thesection{\Alph{section}}
	In this section, we give proofs of Propositions in section 3. We follow \cite{DGR2010}.
	\begin{prop}[Proposition 3.1]\label{prop3.1}
		Let
		\begin{equation}
			L = \left(
			\begin{array}{cc}
				1 & 0 \\
				\lambda^{-1} \log{\frac{z}{2}} & 1
			\end{array}
			\right)
			\left(
			\begin{array}{cc}
				y_0 & \lambda z (y_0)_z \\
				\frac{1}{4} y_1 & y_0 + \lambda z \frac{1}{4} (y_1)_z
			\end{array}
			\right), \nonumber
		\end{equation}
		where
		\begin{equation}
			y_0(z,\lambda) = \sum_{j \ge 0} \frac{\lambda^{-2j} z^{4j} }{16^j(j!)^2} ,\ \ y_1(z,\lambda) = -2 \lambda^{-1} \sum_{j \ge 1} (1 + \cdots + \frac{1}{j} ) \frac{\lambda^{-2j} z^{4j} }{16^j (j!)^2}. \nonumber  
		\end{equation}
		Then \(L \) is a solution of
		\begin{equation}
			L^{-1} dL = \frac{1}{\lambda} \left(
			\begin{array}{cc}
				0 & z^3 \\
				z^{-1} & 0
			\end{array}
			\right)dz, \nonumber
		\end{equation}
		on \(\mathbb{C}^* \). Here, we regard \(\log{z}\) as a multi-valued function on \(\mathbb{C}^*\).
	\end{prop}
	\begin{proof}
		Let \(f_0,  \lambda^{-1} f_0 \log{z} + f_1 \) be fundamental solutions of
		\begin{equation}
			f'' + z^{-1}f' - \lambda^{-2} z^{-1} f = 0, \nonumber
		\end{equation}
		where
		\begin{equation}
			f_0(z) = \sum_{i \ge 0} \frac{z^i}{(i!)^2 \lambda^{2i} },\ \ f_1(z) = -2 \lambda^{-1} \sum_{i \ge 1} (1 + \cdots + \frac{1}{i} ) \frac{z^i}{(i!)^2  \lambda^{2i}}. \nonumber  
		\end{equation}
		We put \(y_0(z) = f_0(\frac{z^4}{16} ) \), then \(y_0'(z) =  \frac{z^3}{4} f_0'(\frac{z^4}{16} ) \) and
		\begin{align}
			y_0''(z) &= \frac{z^6}{16} f_0''(\frac{z^4}{4} ) + \frac{3}{4} z^2 f_0'(\frac{z^4}{16} ) \nonumber \\
			&= \frac{z^6}{16} (- 16 z^{-4} f_0'(\frac{z^4}{16} ) + \lambda^{-2} 16 z^{-4} f(\frac{z^4}{16}) ) + \frac{3}{4} z^2 f_0'(\frac{z^4}{16} )  \nonumber \\
			&= - \frac{1}{4} z^2 f_0'(\frac{z^4}{16} ) + \lambda^{-2} z^2 f(\frac{z^4}{16} ) \nonumber \\
			&= - z^{-1} y_0(z) + \lambda^{-2} z^2 y_0(z). \nonumber 
		\end{align}
		Similarly, \(\lambda^{-1} f_0(\frac{z^4}{16} ) \log{\frac{z}{2} } + \frac{1}{4} f_1(\frac{z^4}{16}  ) \) satisfies \(f'' + z^{-1} f - \lambda^{-2} z^2 f = 0  \). \\
		Thus, 
		\begin{equation}
			L^{-1} dL = \frac{1}{\lambda} \left(
			\begin{array}{cc}
				0 & z^3 \\
				z^{-1} & 0
			\end{array}
			\right)dz. \nonumber
		\end{equation}
		\qed
	\end{proof}
	\begin{lem}
		Let
		\begin{equation}
			E = \left(
			\begin{array}{cc}
				\sqrt{-a} & -\lambda / \sqrt{-a} \\
				0 & 1 / \sqrt{-a}
			\end{array}
			\right) \left(
			\begin{array}{cc}
				1 & 0 \\
				\lambda^{-1} \log{\frac{z}{2}} & 1
			\end{array}
			\right) , \nonumber
		\end{equation}
		where \(a \in \mathbb{R}_{ > 0 } \).\\ Then for any \(e^{i\theta} \in S^1  \), there exists some \(\delta(\theta) \in {\rm SU}_{1,1} \) such that \(E(z e^{i\theta }, \lambda e^{i2\theta } ) = \delta(e^{i\theta}) T(\theta)^{-1} E(z,\lambda) T(\theta) \), where
		\begin{equation}
			T(\theta) = \left(
			\begin{array}{cc}
				e^{-i\theta } & 0 \\
				0 & e^{i\theta }
			\end{array}
			\right). \nonumber
		\end{equation}
	\end{lem}
	\begin{proof}
		Since
		\begin{align}
			(E^{-1} dE)(ze^{i\theta} , \lambda e^{i2\theta } ) &= \frac{1}{\lambda e^{i2\theta} } \left(
			\begin{array}{cc}
				0 & 0 \\
				1 & 0
			\end{array}
			\right) \frac{dz}{z} = \frac{1}{\lambda } \left(
			\begin{array}{cc}
				0 & 0 \\
				1 & 0
			\end{array}
			\right) \frac{dz}{z} \cdot T \nonumber \\
			&= (E^{-1} dE)(z,\lambda) \cdot T, \nonumber
		\end{align}
		there exists a \(\delta = \delta(\theta) \in {\rm SL}_2 \mathbb{C}  \) such that \(E(ze^{i\theta}, \lambda e^{i2\theta}) = \delta E(z,\lambda) T \). Then
		\begin{align}
			\delta &= E(ze^{i\theta}, \lambda e^{i2\theta} ) T^{-1} E(z,\lambda)^{-1} T \nonumber \\
			&= \left(
			\begin{array}{cc}
				\sqrt{-a} & -\lambda e^{i2\theta} / \sqrt{-a} \\
				0 & 1 / \sqrt{-a}
			\end{array}
			\right) \left(
			\begin{array}{cc}
				1 & 0 \\
				\lambda^{-1} e^{-i2\theta} \log{\frac{ze^{i\theta}}{2}} & 1
			\end{array}
			\right) T^{-1} \nonumber \\
			&\ \ \ \ \ \ \ \ \ \ \ \ \ \ \ \ \ \left(
			\begin{array}{cc}
				1 & 0 \\
				-\lambda^{-1} \log{z} & 1	
			\end{array}
			\right) \left(
			\begin{array}{cc}
				1 / \sqrt{-a} & \lambda / \sqrt{-a} \\
				0 & 
				\sqrt{-a}
			\end{array}
			\right) T. \nonumber 
		\end{align}
		and
		\begin{equation}
			\delta \left(
			\begin{array}{cc}
				1 & 0 \\
				0 & -1
			\end{array}
			\right) \delta^* = \left(
			\begin{array}{cc}
				1 & 0 \\
				0 & -1
			\end{array}
			\right). \nonumber
		\end{equation}
		Thus, we obtain \(\delta \in (\Lambda {\rm SU}_{1,1})_{\sigma} \). \qed
	\end{proof}
	\begin{lem}
		For any \(a \in \mathbb{R}_{>0} \), the Iwasawa factorization \(E = F_E w B_E \) near \(z = 0 \) is given by
		\begin{equation}
			w = \left(
			\begin{array}{cc}
				0 & \lambda \\
				-\lambda^{-1} & 0
			\end{array}
			\right), \ \ B_E = \left(
			\begin{array}{cc}
				\sqrt{-a - t - \overline{t}} & - \lambda / \sqrt{-a - t - \overline{t}} \\
				0 & 1 / \sqrt{-a - t - \overline{t}}
			\end{array}
			\right), \nonumber
		\end{equation}
		where \(t = \log{\frac{z}{2}} \).
	\end{lem}
	\begin{proof}
		Since
		\begin{align}
			F_E \left(
			\begin{array}{cc}
				1 & 0 \\
				0 & -1
			\end{array}
			\right) F_E^* &= E B_E^{-1} w^{-1} \left(
			\begin{array}{cc}
				1 & 0 \\
				0 & -1
			\end{array}
			\right) w B_E^{* -1} E^* \nonumber \\
			&= \left(
			\begin{array}{cc}
				1 & 0 \\
				0 & -1
			\end{array}
			\right), \nonumber
		\end{align}
		we obtain \(F_E \in (\Lambda {\rm SU}_{1,1,} )_{\sigma} \) and then \(E = F_E w B_E \).\qed
	\end{proof}
	\begin{prop}[Proposition 3.2]
		For any \(a \in \mathbb{R}_{>0} \) and \(y_0, y_1 \) as in Proposition 3.1,
		\begin{equation}
			\phi = \left(
			\begin{array}{cc}
				\sqrt{-a} & -\lambda / \sqrt{-a} \\
				0 & 1 / \sqrt{-a}
			\end{array}
			\right)
			\left(
			\begin{array}{cc}
				1 & 0 \\
				\lambda^{-1} \log{\frac{z}{2}} & 1
			\end{array}
			\right) \left(
			\begin{array}{cc}
				y_0 & \lambda z (y_0)_z \\
				\frac{1}{4} y_1 & y_0 + \lambda z \frac{1}{4} (y_1)_z
			\end{array}
			\right), \nonumber
		\end{equation}
		admits an Iwasawa factorization
		\begin{equation*}
			\phi = F \left(
			\begin{array}{cc}
				0 & \lambda \\
				-\lambda^{-1} & 0
			\end{array}
			\right) B,\ \ \ F \in (\Lambda {\rm SU}_{1,1})_{\sigma},\ B \in (\Lambda^{+,>0} {\rm SL}_2 \mathbb{C})_{\sigma}, \nonumber
		\end{equation*}
		on some domain \(U = V \cap \mathbb{C} \backslash  (-\infty,0] \), where \(V \) is a neighbourhood of \(z = 0 \) in \(\mathbb{C} \).
	\end{prop}
	\begin{proof}
		We put 
		\begin{equation}
			H = E^{-1} 
			\left( 	
			\begin{array}{cc}
				1 & 0 \\ 0 & -1 
			\end{array}
			\right) E^{-1*} = \left(
			\begin{array}{cc}
				0 & \lambda \\ \lambda^{-1} & -t-\overline{t} -
				a
			\end{array}
			\right). \nonumber
		\end{equation}
		Since \(L_0(0) = Id \) and \(\lim_{z \rightarrow 0} z\log{|z|} = 0 \), \(\lim_{ z \rightarrow 0 } H^{-1} L_0^{-1} H L_0^{-1 *} = Id\).
		Hence we can split \(H^{-1} L_0^{-1} H L_0^{-1 *} = U_{+} U_{-} \) in a neighbourhood of \(z = 0 \) by Birkhoff factorization, where \(U_{+}(0) = U_{-}(0) =Id  \). 
		\begin{align}
			L_0^{-1} H L_0^{-1 *} &= H U_{+} U_{-} = E^{-1} 
			\left( 	
			\begin{array}{cc}
				1 & 0 \\ 0 & -1 
			\end{array}
			\right) E^{-1*} U_{+} U_{-} \nonumber\\
			&= B_E^{-1} \left(
			\begin{array}{cc}
				-1 & 0 \\ 0 & 1
			\end{array}
			\right) B_E^{-1 *} U_{+} U_{-}. \nonumber
		\end{align}
		Now \(B_E = B_E^{(1)} B_E^{(2)}\), where
		\begin{equation}
			B_E^{(1)} = 
			\left(
			\begin{array}{cc}
				\sqrt{-a - t - \overline{t} } & 0 \\ 0 & \frac{1}{\sqrt{-a - t  - \overline{t}} }
			\end{array}
			\right), \ \
			B_E^{(2)} = \left(
			\begin{array}{cc}
				1 & \frac{ \lambda }{-a - t - \overline{t}} \\ 0 & 1
			\end{array}
			\right). \nonumber
		\end{equation}
		Since \(\lim_{z \rightarrow 0} B_E^{(2)} U_{+}  = Id \), there exists a Birkhoff factorization of  \( (B_E^{(2)})^{ -1 *} U_+  \) in a neighbourhood of \(z = 0\). We also split \(B_E^{-1 *} U_+ = V_{+} V_{-} \) by Birkhoff factorization in a neighbourhood of \(z = 0 \) and hence  
		\begin{equation}
			L_0^{-1} H L_0^{-1 *} = B_E^{-1} \left(
			\begin{array}{cc}
				-1 & 0 \\ 0 & 1
			\end{array}
			\right) V_+ V_- U_-, \nonumber
		\end{equation}  
		near \(z = 0 \). We put
		\begin{equation}
			W_+ = B_E^{-1} \left(
			\begin{array}{cc}
				-1 & 0 \\ 0 & 1
			\end{array}
			\right)V_+,\ \ \ W_- = V_- U_-. \nonumber
		\end{equation}
		Since
		\begin{equation}
			W_-^* W_+^* = (L_0^{-1} H L_0^{-1 *} )^* = L_0^{-1} H L_0^{-1 *} = W_+ W_-, \nonumber
		\end{equation}
		\(W_-^* W_0^* = W_+ , \ W_0^{-1 *} W_+^*  = W_-  \), where 
		\begin{equation}
			W_+|_{\lambda = 0 } = W_0 = \left(
			\begin{array}{cc}
				-\rho^{-2} & 0 \\ 0 & \rho^2
			\end{array}
			\right), \nonumber
		\end{equation}
		\(\rho \in \mathbb{R}_{>0} \). We put \(F = \phi B^{-1} w^{-1} \), where
		\begin{equation}
			B = \left(
			\begin{array}{cc}
				-\rho^{-1} & 0 \\ 0 & \rho
			\end{array}
			\right) W_+^{-1}. \nonumber
		\end{equation}
		Then
		\begin{align}
			F \left(
			\begin{array}{cc}
				1 & 0 \\ 0 & -1
			\end{array}
			\right) F^* &= E L_0 B^{-1} w^{-1}
			\left(
			\begin{array}{cc}
				1 & 0 \\ 0 & -1
			\end{array}
			\right) 
			w B^{-1 *} L_0^* E^* \nonumber\\
			&= E L_0 W_+ W_- L_0^* E^* \nonumber\\
			&= E H E^* \nonumber\\
			&= \left(
			\begin{array}{cc}
				1 & 0 \\ 0 & -1
			\end{array}
			\right). \nonumber
		\end{align}
		Hence \(F \in ( \Lambda {\rm SU}_{1,1})_{\sigma},\ B \in (\Lambda^{+,>0} {\rm SL}_2 \mathbb{C} )_{\sigma} \) and we can split \(\phi = F w B  \) near \(z=0 \) via Iwasawa factorization.\qed
	\end{proof}
	\begin{prop}[Proposition 3.3]
		For any \(e^{i\theta } \in S^1 \),
		\begin{equation}
			\phi(z e^{i\theta}, \lambda e^{i2\theta} ) = \delta T^{-1} \phi(z,\lambda) T, \nonumber
		\end{equation}
		where \(\delta = \delta(\theta) \in (\Lambda {\rm SU}_{1,1 } )_{\sigma } \) and
		\begin{equation}
			T = T(\theta) = \left( \begin{array}{cc}
				e^{-i\theta } & 0 \\
				0 & e^{i\theta }
			\end{array} \right) \in U(1). \nonumber
		\end{equation} 
	\end{prop}
	\begin{proof}
		\begin{equation}
			L_0^{-1} dL_0 + \frac{1}{\lambda} L_0^{-1} N L_0 \frac{dz}{z} = L_0^{-1} dL_0 = \eta, \nonumber
		\end{equation}
		with the initial condition \(L_0(0) = Id \).  Since \(T L_0(ze^{i\theta },\lambda e^{i 2\theta }) T^{-1} \) satisfies the above equation, \( L_0(ze^{ i\theta },\lambda e^{i 2\theta }) = T^{-1} L_0(z,\lambda) T  \). Thus
		\begin{align}
			\phi(ze^{i \theta},\lambda e^{i 2\theta}) &=
			E(z e^{ i\theta},\lambda e^{i 2\theta}) L_0(z e^{ i\theta},\lambda e^{i 2\theta})
			= \delta T^{-1} E(z,\lambda) T T^{-1} L_0 T \nonumber \\
			&= \delta T^{-1} \phi(z,\lambda) T. \nonumber
		\end{align}
		\qed
	\end{proof}
	\begin{cor}[Corollary 3.4]\label{cor3.4}
		Let
		\begin{equation}
			\phi = F \left(
			\begin{array}{cc}
				0 & \lambda \\
				-\lambda^{-1} & 0
			\end{array}
			\right) B, \nonumber
		\end{equation}
		be an Iwasawa factorization on some domain \(U = V \cap \mathbb{C} \backslash  (-\infty,0] \), where \(V \) is a neighbourhood of \(z = 0 \) in \(\mathbb{C} \). We put
		\begin{equation}
			B(z,\lambda = 0) = \left(
			\begin{array}{cc}
				e^{\frac{v}{2} } & 0 \\
				0 & e^{- \frac{v}{2} }
			\end{array}
			\right), \nonumber
		\end{equation}
		then \(v(z, \overline{z}) = v(|z|) \) and \(v:U \rightarrow \mathbb{R} \) satisfies
		\begin{equation}
			\frac{1}{4} \left(v_{rr} + \frac{1}{r} v_r  \right) - r^6 e^{2v} + r^{-2} e^{-2v} = 0, \nonumber
		\end{equation}
		and \(e^{\frac{v}{2}} \sim \sqrt{-a -2 \log{r} } \) as \(r \rightarrow 0\), where \(r = |z| \).
	\end{cor}
	\begin{proof}
		We put
		\begin{equation}
			\tilde{F} = F \left(
			\begin{array}{cc}
				0 & \lambda \\
				-\lambda^{-1} & 0
			\end{array}
			\right). \nonumber
		\end{equation}
		Since
		\begin{align}
			\tilde{F}^{-1}d\tilde{F} &= B \xi B^{-1} - dB B^{-1} \nonumber \\
			&= \left(\begin{array}{cc}
				\frac{1}{2}v_z & \lambda^{-1}z^3e^v \\
				\lambda^{-1}z^{-1}e^{-v} & -\frac{1}{2}v_z
			\end{array}\right)dz + \left(\begin{array}{cc}
				-\frac{1}{2}v_{\overline{z}} & \lambda \overline{z}^{-1}e^{-v} \\
				\lambda \overline{z}^3e^{v} & \frac{1}{2}v_{\overline{z}}
			\end{array}\right)d\overline{z}, \nonumber
		\end{align}
		and \(\tilde{F}^{-1}d\tilde{F}\) satisfies the zero curvature equation, we obtain
		\begin{equation}
			\frac{1}{4} \left(v_{rr} + \frac{1}{r} v_r  \right) - r^6 e^{2v} + r^{-2} e^{-2v} = 0. \nonumber
		\end{equation}
		From the proof of Proposition A.2, we obtain \(e^{\frac{v}{2}} \sim \sqrt{-a -2 \log{r} } \) as \(r \rightarrow 0\).
		\qed
	\end{proof}
	\section{Bessel functions}
	\def\thesection{\Alph{section}}
	In this section, we review the definition of the Bessel functions and their asymptotic expansions. We refer to section 17 of \cite{WW1996}. For any real number \(\alpha \), the Bessel equation is
	\begin{equation}
		x^2 \frac{d^2 y}{d x^2} + x \frac{d y}{dx} + \left(x^2 - \alpha^2 \right) y = 0,\ \ \ x \in \mathbb{C}. \nonumber 
	\end{equation}
	The Bessel functions are canonical solutions \(y(x) \) of the Bessel equation and we call \(\alpha \) the order of the Bessel function. The Bessel functions are defined as follows.
	\begin{Def}
		For \(\alpha \notin \mathbb{Z} \), \(x \in \mathbb{C} \)
		\begin{align}
			&J_{\alpha }(x) = \sum_{j = 0 }^{\infty }\frac{(-1)^j }{j! (j + \alpha + 1) } \left(\frac{x}{2} \right)^{2j+\alpha } \ : \ {\rm Bessel\ function\ of\ the\ first\ kind }, \nonumber \\
			&Y_{\alpha}(x) = \frac{J_{\alpha}(x) \cos{(\alpha \pi )} - J_{-\alpha}(x) }{\sin{(\alpha \pi)}}\ : \ \begin{gathered}
				{\rm Bessel\ function\ of\ the\ second\ kind} \\
				\ \ \ ({\rm Neumann\ function})
			\end{gathered}, \nonumber \\
			&H_{\alpha}^{(1)}(x) = J_{\alpha }(x) + i Y_{\alpha}(x) \ : \ {\rm Hankel\ function\ of\ the\ first\ kind}, \nonumber \\
			&H_{\alpha}^{(2)}(x) = J_{\alpha }(x) - i Y_{\alpha}(x) \ : \ {\rm Hankel\ function\ of\ the\ second\ kind}, \nonumber
		\end{align} \vspace{2mm}
		For \(n \in \mathbb{Z}, x \in \mathbb{C} \), we define the Bessel functions by
		\begin{align}
			&J_n(x) = \lim_{\alpha \rightarrow n} J_{\alpha }(x),\ \ \ Y_n(x) = \lim_{\alpha \rightarrow n} Y_{\alpha }(x), \nonumber \\
			&H_n^{(1)}(x) = \lim_{\alpha \rightarrow n} H_{\alpha }^{(1)}(x),\ \ \ H_n^{(2)}(x) = \lim_{\alpha \rightarrow n} H_{\alpha }^{(2)}(x). \nonumber
		\end{align}
		We remark that \(Y_n, H_n^{(1)}, H_2^{(2)} \) are multi-valued functions on \(\mathbb{C} \).
	\end{Def} \vspace{2mm}
	
	In the Bessel equation, we replace \(x \) by \( ix\) then we obtain the modified Bessel equation
	\begin{equation}
		x^2 \frac{d^2 x}{d x^2} + x \frac{d y}{dx} - \left(x^2 + \alpha^2 \right) y = 0,\ \ \ x \in \mathbb{C}. \nonumber
	\end{equation}
	The canonical solutions of the modified Bessel equation are given as follows.
	\begin{Def}\label{defA2}
		For \(\alpha \notin \mathbb{Z}, x \in \mathbb{C} \), the modified Bessel function are defined by
		\begin{align}
			&I_{\alpha }(x) = \sum_{j = 0 }^{\infty }\frac{1 }{j! (j + \alpha + 1) } \left(\frac{x}{2} \right)^{2j+\alpha } \ : \ 
			\begin{gathered}
				{\rm modified\ Bessel\ function} \\
				\ \ \ \ \ \ \ \ \ \ \ \ \ \ \ \ \ \ {\rm of\ the\ first\ kind}
			\end{gathered}, \nonumber \\
			&K_{\alpha }(x) = \frac{\pi}{2} \frac{I_{-\alpha}(x) - I_{\alpha}(x) }{\sin{(\alpha \pi )} }\ : \ {\rm modified\ Bessel\ function\ of\ second\ kind}, \nonumber
		\end{align}
		and for \(n \in \mathbb{Z} \) we define
		\begin{equation}
			I_n(x) = \lim_{\alpha \rightarrow n} I_{\alpha }(x),\ \ \ K_n(x) = \lim_{\alpha \rightarrow n} K_{\alpha }(x). \nonumber
		\end{equation}
	\end{Def}
	From the definition, we have
	\begin{equation}
		I_{\alpha }(x) = e^{-i \frac{\alpha \pi}{2} } J_{\alpha }(ix),\ \ \ K_{\alpha}(x) = K_{-\alpha}(x) = i\frac{\pi}{2} e^{i \frac{\alpha \pi}{2} } H_{\alpha}^{(1)}(ix). \nonumber
	\end{equation}
	Furthermore, we obtain the following Proposition.
	\begin{prop}\label{propA1}
		For \(n \in \mathbb{Z}_{\ge 0 }, x \in \mathbb{C} \), the expansion series of \(Y_n, K_n \) at \(x = 0 \) are given by
		\begin{align}
			&Y_0(x) = \frac{2}{\pi} J_0(x) \log{\left(\frac{x}{2} \right)} - \frac{2}{\pi} \sum_{j=0}^{\infty} \frac{\psi(j+1)}{(j!)^2} (-1)^j \left( \frac{x}{2} \right)^{2j }, \nonumber \\
			&Y_n(x) = \frac{2}{\pi} J_n(x) \log{\left(\frac{x}{2} \right)} - \frac{1}{\pi} \sum_{j=0}^{n-1} \frac{(n-j-1)!}{j!} \left( \frac{x}{2} \right)^{-n+2j } \nonumber \\
			&\hspace{3.7cm} - \frac{1}{\pi} \sum_{i=0}^{\infty} \frac{\psi(j+1) + \psi(n+j+1)}{j!(n+j)!}(-1)^j \left( \frac{x}{2} \right)^{n+2j }, \nonumber \\
			&K_0(x) = -I_0(x) \log{\left(\frac{x}{2} \right)} + \sum_{j=0}^{\infty} \frac{\psi(j+1)}{(j!)^2} \left(\frac{x}{2} \right)^{2j}, \nonumber \\
			&K_n(x) = \frac{1}{2} \sum_{j=0}^{n-1} \frac{(-1)^j \psi(n-j-1)!}{j!} \left(\frac{x}{2} \right)^{2j-n} \nonumber \\
			&\hspace{1.2cm} + \sum_{j=0}^{\infty} \frac{(-1)^{n+1} }{j!(n+1)!} \left(\frac{x}{2} \right)^{n +2j } \left\{\log{\left(\frac{x}{2} \right)} - \frac{1}{2} \psi(j+1) - \frac{1}{2} \psi(n+j+1) \right\}, \nonumber
		\end{align}
		where \(\psi(1) = - \gamma \) and for \(j \ge 1 \)
		\begin{equation}
			\psi(j+1) = - \gamma + \left( 1 + \frac{1}{2} + \cdots + \frac{1}{j} \right). \nonumber
		\end{equation}
	\end{prop}
	\begin{rem}
		The function \(\psi \) is called the digamma function. For \(x \in \mathbb{C} \), the digamma function is defined by
		\begin{equation}
			\psi(x) = \frac{d}{dx} \log{\Gamma(x)} = \frac{(\Gamma(x))_x}{\Gamma(x)}, \nonumber
		\end{equation}
		where \(\Gamma \) is the gamma function.
	\end{rem}
	Next, we consider the asymptotic expansions of the Bessel functions.
	\begin{prop}
		Fix \(a, n \in \mathbb{Z}_{\ge 0} \). For \(-\pi < {\rm arg}(x) < 2\pi \), we have
		\begin{equation}
			H_{\alpha}^{(1)}(x) = \sqrt{\frac{2 }{\pi x } } e^{i(x - \frac{\pi}{4} - \frac{\alpha \pi}{2})} \left\{ \sum_{j=0}^{n-1} \frac{\Gamma(\alpha + j + 1/2)}{j!\Gamma(\alpha -j + 1/2)} \left(\frac{i}{2x} \right)^j + R_{1,n}^{\alpha}(x) \right\}, \nonumber
		\end{equation}
		\begin{equation}
			H_{\alpha}^{(2)}(x) = \sqrt{\frac{2 }{\pi x } } e^{-i(x - \frac{\pi}{4} - \frac{\alpha \pi}{2})} \left\{ \sum_{j=0}^{n-1} \frac{\Gamma(\alpha + j + 1/2)}{j!\Gamma(\alpha -j + 1/2)} \left(\frac{-i}{2x} \right)^j + R_{2,n}^{\alpha}(x) \right\}, \nonumber
		\end{equation}
		where
		\begin{equation}
			|R_{1,n}^{\alpha}(x)| < C_1 |x|^{-n},\ \ \ |R_{2,n}^{\alpha}(x)| < C_2 |x|^{-n},\ \ \ C_1,\ C_2:{\rm constant}. \nonumber
		\end{equation}
	\end{prop}
	\begin{prop}[\cite{WW1996}]
		Fix \(\alpha, n \in \mathbb{Z}_{\ge 0} \). For \(-\pi < {\rm arg}(x) < \pi \), we have
		\begin{equation}
			J_{\alpha}(x) = \frac{1}{\sqrt{2\pi x}}\left\{e^{i\left(\alpha+\frac{1}{2}\right)\frac{\pi}{2}}W_{0,\alpha}(2ix) + e^{-i\left(\alpha+\frac{1}{2}\right)\frac{\pi}{2}}W_{0,\alpha}(-2ix) \right\}, \nonumber
		\end{equation}
		\begin{equation}
			Y_{\alpha}(x) = \cos{(\alpha \pi)} \frac{e^{-i\alpha \pi}}{\sqrt{2\pi x}}\left\{e^{i\left(\alpha+\frac{3}{2}\right)\frac{\pi}{2}}W_{0,\alpha}(2ix) + e^{-i\left(\alpha+\frac{3}{2}\right)\frac{\pi}{2}}W_{0,\alpha}(-2ix) \right\}, \nonumber
		\end{equation}
		where \(W_{0,\alpha}\) is a solution of
		\begin{equation}
			\frac{d^2W}{dx^2} + \left\{-\frac{1}{4} + \frac{1}{z^2}\left(\frac{1}{4} - \alpha^2\right) \right\}, \nonumber
		\end{equation}
		on \(\mathbb{C}^*\) with the asymptoics expansion
		\begin{equation}
			W_{0,\alpha}(x) \sim e^{-\frac{x}{2}}\left\{1 + \sum_{l=1}^{\infty}\frac{\left\{\alpha^2  - \frac{1}{4}\right\}\left\{\alpha^2 - \frac{9}{4}\right\} \cdots \left\{\alpha^2 - \left(l - \frac{1}{2}\right)^2 \right\}}{l!x^l} \right\}, \nonumber
		\end{equation}
		for large \(|x|\), \(-\pi < {\rm arg}(x) < \pi\). Equivalently, we have
		\begin{equation}
			J_{\alpha}(x) = \sqrt{\frac{2}{\pi x}}\left\{(1+T_1^{\alpha}(x))\cos{\left(x-\frac{\alpha \pi}{2}-\frac{\pi}{4}\right) - T_2^{\alpha}(x)\sin{\left(x-\frac{\alpha \pi}{2}-\frac{\pi}{4} \right)}}
			\right\}, \nonumber
		\end{equation}
		\footnotesize
		\begin{equation}
			Y_{\alpha}(x) = \cos{(\alpha \pi)} e^{-i\alpha \pi}\sqrt{\frac{2}{\pi x}}\left\{
			(1+T_1^{\alpha}(x))\sin{\left(x - \frac{\alpha \pi}{2} - \frac{\pi}{4}\right)} + T_2^{\alpha}(x)\cos{\left(x - \frac{\alpha \pi}{2} -\frac{\pi}{4}\right)}
			\right\}, \nonumber
		\end{equation}
		\normalsize
		where
		\begin{align}
			1 + T_1^{\alpha}(x) &= \frac{1}{2}\left\{e^{ix}W_{0,\alpha}(2ix) + e^{-ix}W_{0,\alpha}(-2ix)\right\} \nonumber\\
			&\sim 1 +\sum_{l=1}^{\infty}(-1)^l\frac{\left\{\alpha^2  - \frac{1}{4}\right\}\left\{\alpha^2 - \frac{9}{4}\right\} \cdots \left\{\alpha^2 - \left(2l - \frac{1}{2}\right)^2 \right\}}{2^{2l}l!(x)^{2l}}, \nonumber
		\end{align}
		\begin{align}
			T_2^{\alpha}(x) &= \frac{i}{2}\left\{e^{ix}W_{0,\alpha}(2ix) - e^{-ix}W_{0,\alpha}(-2ix)\right\} \nonumber\\
			&\sim \sum_{l=1}^{\infty}(-1)^{l-1}\frac{\left\{\alpha^2  - \frac{1}{4}\right\}\left\{\alpha^2 - \frac{9}{4}\right\} \cdots \left\{\alpha^2 - \left(2l - \frac{3}{2}\right)^2 \right\}}{2^{2l-1}l!(x)^{2l-1}}, \nonumber
		\end{align}
		for large \(|x|\), \(-\pi < {\rm arg}(x) < \pi\).
	\end{prop}

	\subsection*{Acknowledgement}
	The author would like to thank Professor Martin Guest for his considerable support. The author would also like to thank the members of the geometry group at Waseda for their helpful comments and discussion. 

	\bibliography{mybibfile}

\begin{thebibliography}{10}

\bibitem{BD2001}
V.~Balan and J.~Dorfmeister.
\newblock {Birkhoff decomposition and Iwasawa decomposition for loop groups}.
\newblock {\em T$\hat{o}$hoku Math. J.}, 53(4), 2001.

\bibitem{BI1995}
B.~Bobenko and A.~Its.
\newblock {The Painlev\'e III equation and the Iwasawa decomposition}.
\newblock {\em Manuscripta Math.}, 87(3), 1995.

\bibitem{BRS2010}
D.~Brander, W.~Rossman, and N.~Schmitt.
\newblock {Holomorphic representation of constant mean curvature surfaces in
  Minkowski space: consequences of non-compactness in loop group methods}.
\newblock {\em Adv. Math.}, 223(3):949--986, 2010.

\bibitem{CV1991}
S.~Cecotti and C.~Vafa.
\newblock Topological-anti-topological fusion.
\newblock {\em Nuclear Phys. B}, 367(2):359--461, 1991.

\bibitem{DGR2010}
J.~Dorfmeister, M.~Guest, and W.~Rossman.
\newblock {The tt* structure of the quantum cohomology of $\mathbb C$$P^1$ from
  the viewpoint of differential geometry}.
\newblock {\em Asian J. Math.}, 14(3):417--437, 2010.

\bibitem{DH1998}
J.~Dorfmeister and G.~Haak.
\newblock On symmetries of constant mean curvature surfaces.
\newblock {\em Tôhoku Math. J.}, 50:437--454, 1998.

\bibitem{DH2000}
J.~Dorfmeister and G.~Haak.
\newblock On symmetries of constant mean curvature surfaces. {II}. symmetries
  in a weierstrass-type representation.
\newblock {\em Int. J. Math. Game Theory algebra}, 10:121--146, 2000.

\bibitem{DPW1998}
J.~Dorfmeister, F.~Pedit, and H.~Wu.
\newblock Weierstrass type representation of harmonic maps into symmetric
  spaces.
\newblock {\em Comm. Anal. Geom.}, 6:633--668, 1998.

\bibitem{E1993}
J.~Eells.
\newblock {The Surfaces of Delaunay}.
\newblock {\em Math. Intelligencer}, 214:527--565, 1993.

\bibitem{FIKN2006}
A.~Fokas, A.~Its, A.~Kapaev, and V.~Novokshenov.
\newblock {\em {Painlev\'e Transcendents: The Riemann-Hilbert Approach,
  Mathematical}}.
\newblock Surveys and Monographs 128. Amer. Math Soc., 2006.

\bibitem{GIL20151}
M.~Guest, A.~Its, and C.~Lin.
\newblock Isomonodromy aspects of the tt* equations of Cecotti and Vafa I. Stokes data.
\newblock {\em Int. Math. Res. Notices}, 2015(22), 2015.

\bibitem{GIL20152}
M.~Guest, A.~Its, and C.~Lin.
\newblock {Isomonodromy aspects of the tt* equations of Cecotti and Vafa II.
  Riemann-Hilbert problem}.
\newblock {\em Comm. Math. Phys.}, 336, 2015.

\bibitem{GIL2020}
M.~Guest, A.~Its, and C.~Lin.
\newblock {Isomonodromy aspects of the tt* equations of Cecotti and Vafa III.
  Iwasawa factorization and asymptotics}.
\newblock {\em Comm. Math. Phys.}, 374(2), 2020.

\bibitem{GL2014}
M.~Guest and C.~Lin.
\newblock {Nonlinear PDE aspects of the tt* equations of Cecotti and Vafa}.
\newblock {\em J. Reine Angew. Math.}, 689, 2014.

\bibitem{IN2011}
A.~Its and D.~Niles.
\newblock {On the Riemann-Hilbert-Birkhoff inverse monodromy problem associated
  with the third Painlev\'e equation}.
\newblock {\em Lett. Math. Phys.}, 96, 2011.

\bibitem{K2004}
M.~Kilian.
\newblock {On the associated family of Delaunay surfaces}.
\newblock {\em Proc. Amer. Math. Soc.}, 132(10), 2004.

\bibitem{O2017}
Y.~Ogata.
\newblock {The DPW method for constant mean curvature surfaces in
  3--dimensional Lorentzian spaceforms, with applications to Smyth type
  surfaces}.
\newblock {\em Hokkaido Math. J.}, 46(3), 2017.

\bibitem{PS1986}
A.~Presseley and G.~Segal.
\newblock {\em {Loop Groups}}.
\newblock Oxford Univ. Press., 1986.

\bibitem{SKKR2007}
N.~Schmitt, M.~Kilian, S.-P. Kobayashi, and W.~Rossman.
\newblock {Unitarization of monodromy representations and constant mean
  curvature trinoids in 3-dimensional space forms}.
\newblock {\em J. Lond. Math. Soc. (2)}, 75(3), 2007.

\bibitem{S1993}
B.~Smyth.
\newblock {A generalization of a theorem of Delaunay on constant mean curvature
  surfaces}.
\newblock {\em IMA Vol. Math. Appl.}, 51, 1993.

\bibitem{TPF1994}
M.~Timmreck, U.~Pinkall, and D.~Ferus.
\newblock {Constant mean curvature places with inner rotational symmetry in
  Euclidean 3--space}.
\newblock {\em Math. Z.}, 215(4), 1994.

\bibitem{WW1996}
E.~Whittaker and G.~Watson.
\newblock {\em {A course of modern analysis. Reprint of the fourth (1927)
  edition}}.
\newblock Cambridge Mathematical Library. Cambridge University Press,
  Cambridge, 1996.

\end{thebibliography}
	\bibliographystyle{plain}
	
	\em
	\noindent
	Department of Mathematics\newline
	Faculty of Science and Engineering\newline
	Waseda University\newline
	3-4-1 Okubo, Shinjuku, Tokyo 169-8555\newline
	JAPAN

\end{document}